\documentclass[reqno]{amsart}

\usepackage[utf8]{inputenc}
\usepackage[finnish,english]{babel}

\usepackage{xcolor}
\usepackage{hyperref}
\usepackage{xcolor}
\usepackage{amssymb}
\usepackage{mathtools}
\usepackage{amsaddr}

\usepackage{tikz}
\usetikzlibrary{calc}

\usepackage{soul}

\usepackage{hyperref}

\usepackage[pagewise]{lineno}%\linenumbers

\usepackage[shortlabels]{enumitem}

\usepackage{accents}

\theoremstyle{plain}
\newtheorem{theorem}{Theorem}
\newtheorem{proposition}[theorem]{Proposition}
\newtheorem{lemma}[theorem]{Lemma}

\theoremstyle{definition}

\newtheorem{remark}[theorem]{Remark}

\let\d\undefined
\let\H\undefined

\newcommand*{\N}{\mathbb{N}}

\newcommand*{\R}{\mathbb{R}}

\newcommand*{\d}{\mathrm{d}}
\newcommand{\der}{\mathrm{d}}
\newcommand*{\norm}[1]{\left\lVert#1\right\rVert}
\newcommand{\ip}[2]{\left\langle#1,#2\right\rangle}
\newcommand{\iip}[2]{\left(#1,#2\right)}
\newcommand{\H}{\mathcal{H}}
\newcommand{\V}{\mathcal{V}}

\newcommand{\eps}{\varepsilon}

\newcommand{\doo}[1]{\partial_{\mathrm{#1}}}

\newcommand{\lap}[1]{\overset{\mathtt{#1}}{\Delta}}
\newcommand{\grad}[1]{\overset{\mathtt{#1}}{\nabla}}
\newcommand{\dive}[1]{\overset{\mathtt{#1}}{\operatorname{div}}}
\newcommand{\abs}[1]{\left\vert#1\right\vert}

\newcommand{\alf}[1]{\accentset{\alpha}{#1}}

\newcommand{\C}[3]{C^{#1}_{\mathtt{h}}C^{#2}_{\mathtt{v}}(#3)}

\newcommand{\HS}[3]{H^{#1}_{\mathtt{h}}H^{#2}_{\mathtt{v}}(#3)}

\newcommand{\OS}[3]{\Omega^{#1}_{\mathtt{h}}\Omega^{#2}_{\mathtt{v}}(#3)}

\newcommand{\LS}[3]{\Lambda^{#1}_{\mathtt{h}}\Lambda^{#2}_{\mathtt{v}}(#3)}
\DeclareMathOperator{\supp}{supp}

\DeclareMathOperator{\Lip}{Lip}

\DeclareMathOperator{\sisus}{int}

\DeclareMathOperator{\trace}{tr}

\newcommand{\NTR}[1]{}

\title[Tensor tomography on manifolds of low regularity]{Tensor tomography on negatively curved manifolds of low regularity}

\author{Joonas Ilmavirta}
\address{Department of Mathematics and Statistics\\
University of Jyv\"askyl\"a\\
P.O. Box 35 (MaD)\\
FI-40014 University of Jyv\"askyl\"a, Finland\\
\texttt{joonas.ilmavirta@jyu.fi}}
\author{Antti Kykkänen}
\address{Department of Mathematics and Statistics\\
University of Jyv\"askyl\"a\\
P.O. Box 35 (MaD)\\
FI-40014 University of Jyv\"askyl\"a, Finland\\
\texttt{antti.k.kykkanen@jyu.fi}}

\date{\today}


\begin{document}

\maketitle

\begin{abstract}
We prove solenoidal injectivity for the geodesic X-ray transform of tensor fields on simple Riemannian manifolds with $C^{1,1}$ metrics and non-positive sectional curvature.
The proof of the result rests on Pestov energy estimates for a transport equation on the non-smooth unit sphere bundle of the manifold.

Our low regularity setting requires keeping track of regularity and making use of many functions on the sphere bundle having more vertical than horizontal regularity.
Some of the methods, such as boundary determination up to gauge and regularity estimates for the integral function, have to be changed substantially from the smooth proof.
The natural differential operators such as covariant derivatives are not smooth.
\end{abstract}

%%%%
%%%%
%%%%

\section{Introduction}

\NTR{We have indicated all changes to the manuscript with these footnotes.}

What are the minimal smoothness assumptions on a Riemannian metric under which the geodesic X-ray transform of tensor fields on the Riemannian manifold is solenoidally injective?
Solenoidal injectivity on smooth simple manifolds with negative curvature was proved in~\cite{PSIGOTFMNC}.
Since~\cite{PSIGOTFMNC}, many solenoidal injectivity results have been shown under different variations of the geometric setup.
Solenoidal injectivity is known for all real analytic simple Riemannian metrics~\cite{SUBRASFGSM} and for all smooth simple Riemannian metrics with certain bounds on their terminator values~\cite{PSUIDBTTT}.
The study of the X-ray transform on manifolds with Riemannian metrics of low regularity was started recently~\cite{IKPIXRTMLR}, where the authors prove that the X-ray transform of scalar functions is injective on all simple manifolds with~$C^{1,1}$ Riemannian metrics.
We extend this result and prove that the X-ray transform of tensor fields of any order is solenoidally injective for all simple~$C^{1,1}$ Riemannian metrics with almost everywhere non-positive sectional curvature.

X-ray tomography problems of $2$-tensor fields naturally arise as linearized problems of travel time tomography or boundary rigidity~\cite{SUVZTTT}.
The travel time problem arises in applications, such as seismological imaging, where one asks whether the sound speed in a medium can uniquely be determined from the knowledge of the arrival times of waves on the boundary.
Because of the geophysical nature of such problems, it is relevant to ask how well the studied model corresponds to the real world.
From this point of view, the smoothness assumption of the model manifold is merely a mathematical convenience, which is why we have set out to relax such assumptions.

Our main objective is to optimize the regularity assumptions imposed on the Riemannian metric~$g$ of the manifold.
We focus on global and uniform non-smoothness (as opposed to, say, interfaces with jump discontinuities), and as in~\cite{IKPIXRTMLR} the natural optimality to aim at remains~$C^{1,1}$.
If~$g$ is only assumed to be in the Hölder space~$C^{1,\alpha}$ for~$\alpha < 1$, the geodesic equation fails to have unique solutions~\cite{HartmanLUG,SSGLR} and the X-ray transform itself becomes ill-defined.
In this sense our result is optimal on the Hölder scale, as we provide a solenoidal injectivity result (theorem~\ref{thm:c11-s-injectivity}) for the class of simple~$C^{1,1}$ Riemannian metrics with almost everywhere non-positive sectional curvature.

The non-positivity assumption on the curvature is likely unnecessary --- milder assumptions on top of simplicity could suffice.
However, even in the smooth case relaxing the curvature assumption causes technical difficulties and solenoidal injectivity for all simple Riemannian metrics is not understood.
Since our setting is complicated enough as it is, we decided not to include manifolds with possible positive curvature.

A popular method for proving injectivity results relies on interplay between the X-ray transform and a transport equation.
In the smooth case, the transport equation is studied using the so called Pestov identity and energy estimates derived from it (see e.g.~\cite{PSUGIPETD,PSUTTPC,IMIGMBA} and references therein).

We employ a similar approach in our non-smooth setting.
Our proof is structurally the same as those in smooth geometry, so the main content of this article is to ensure that everything is well defined and behaved in our non-smooth setting:
the unit sphere bundle and operators on it, commutator formulas, function spaces, Santal\'o's formula, and others.

%%%%
%%%%
%%%%
\subsection{Main results}

We record as our main result the following kernel description for the geodesic X-ray transform of tensor fields.
In the literature of the geodesic X-ray transform similar results are often called solenoidal injectivity results.
Throughout the article~$M$ will be a compact and connected smooth manifold with a smooth boundary~$\partial M $.
The dimension of~$M$ will always be~$n \ge 2$.
The manifold~$M$ comes equipped with a~$C^{1,1}$ regular Riemannian metric~$g$.
That is, the metric~$g$ is continuously differentiable and the derivative is Lipschitz.

We define what it means for~$(M,g)$ to be simple in section~\ref{sec:definitions-notation}.
Simple~$C^{1,1}$ manifolds have global coordinates by definition, but for smooth simple manifolds this is a consequence of the definitions.
When $g \in C^\infty$ the definition of $C^{1,1}$ simplicity is equivalent to the classical definition~\cite[Theorem 2]{IKPIXRTMLR} and thus assuming existence of global coordinates is not superfluous.
We say that $g$ has almost everywhere non-positive sectional curvature if for almost all $x \in M$ we have $\ip{R(w,v)v}{w}_{g(x)} \le 0$ where $v,w \in T_xM$ are orthogonal.
The curvature tensor~$R$ is well-defined by the familiar formula almost everywhere in~$M$.
The X-ray transform of tensor fields is defined in section~\ref{subsubsec:transform}.

\begin{theorem}
\label{thm:c11-s-injectivity}
Let~$(M,g)$ be a simple~$C^{1,1}$ manifold (see section~\ref{sec:definitions-notation}) with almost everywhere non-positive sectional curvature.
Let~$m \ge 1$ be an integer.
\begin{enumerate}
\item\label{item:thm1} If~$p \in C^{1,1}(M)$ is a symmetric~$(m-1)$-tensor field vanishing on~$\partial M$, then the X-ray transform $I(\sigma\nabla p)$ of its symmetrized covariant derivative vanishes.

\item\label{item:thm2} If the X-ray transform~$If$ of~a symmetric~$m$-tensor field~$f \in C^{1,1}(M)$ vanishes, there is a symmetric~$(m-1)$-tensor field~$p \in \Lip(M)$ vanishing on~$\partial M$ so that~$f = \sigma\nabla p$ almost everywhere on $M$.
\end{enumerate}
\end{theorem}

\subsection{Regularity discussion}

Claims~\ref{item:thm1} and~\ref{item:thm2} in theorem~\ref{thm:c11-s-injectivity} are not symmetric.
The difference is in the regularity of the potential~$p$ and we believe this is only a consequence of our proof techniques.

There are two notions of smoothness of any given order of a tensor field:
regularity with respect to the smooth structure and existence of high order covariant derivatives.
The covariant concept of smoothness is more natural on a Riemannian manifold.
For a typical tensor field~$f$ that is~$C^\infty$ smooth in the sense of the smooth structure, the covariant derivative~$\nabla f$ is typically only Lipschitz when $g\in C^{1,1}$.
The metric tensor~$g$ and its tensor powers are examples of non-vanishing and non-smooth (in the sense of the smooth structure) tensor fields for which covariant derivatives of all orders are well defined.
Thus neither of the two notions of smoothness implies the other in general.
The two notions of~$C^{1,1}$ and less regular H\"older spaces of tensor fields agree, but they disagree for higher regularity.
Therefore there are, for example, two different spaces~$C^{2,1}$ and we do not use such confusing spaces at all.

We focus on optimizing the regularity of the Riemannian metric~$g$, but we did not pursue optimizing regularity of the tensor fields~$f$ or~$p$, the boundary~$\partial M$ or the integral function~$u^f$ of~$f$ (see equation~\eqref{eq:function-uf}).

It is important for our key regularity result (lemma~\ref{lma:regularity-u} below) that the boundary values of the tensor field are determined by the data to the extent allowed by gauge freedom.
A boundary determination result for~$2$-tensor fields in the smooth case, where~$g$ is~$C^\infty$, can be found in~\cite[Lemma 4.1]{SUBRASFGSM}.
Their result is based on clever analysis of equation~$2f_{ij} = p_{i;j} + p_{j;i}$ in boundary normal coordinates.
Although the argument in~\cite{SUBRASFGSM} works nicely in the smooth case, it does not give the desired result if~$g$ is only~$C^{1,1}$ and~$f$ is~$C^{1,1}$.
The immediate conclusion of their argument in the non-smooth case would be that~$p$ has derivatives in some directions and is Lipschitz continuous, whereas in lemma~\ref{lma:boundary-determination} we find a~$p$ in the class~$C^{1,1}$.
The other difficulty in adapting similar arguments to the non-smooth case is the regularity of boundary normal coordinates.

To avoid these issues we prove a boundary determination result (lemma~\ref{lma:boundary-determination}) by a more explicit approach.
Our construction gives a potential~$p \in C^{1,1}(M)$ satisfying $\sigma\nabla p|_{\partial M} = f|_{\partial M}$ when~$f \in C^{1,1}(M)$.
The cost of our method compared to the method of~\cite{SUBRASFGSM} is losing control of the $1$-jets in any neighbourhood of the boundary, but leading order boundary determination suffices for our needs.

We lose a derivative in the regularity of~$p$ twice in our argument:
\begin{enumerate}
    \item \label{item:deriv-loss-1} We lose a derivative of~$p$ in the boundary determination result. Even if the tensor field~$f \in C^{l,1}(M)$ and the Riemannian metric~$g \in C^{k,1}(M)$ are assumed to have any (finite) amounts of derivatives, we only get~$p \in C^{\min(k,l),1}(M)$.
    Particularly,~$p$ is only~$C^{1,1}$, when~$g$ and~$f$ are~$C^{1,1}$.
    To our knowledge, our boundary determination result is optimal in the literature for differentiability of the potential~$p$ with properties~$\sigma\nabla p = f$ and~$p = 0$ on the boundary.
    
    One might expect~$f|_{\partial M} = \sigma\nabla p|_{\partial M}$, where~$f \in C^{1,1}(M)$ and~$p \in C^{2,1}(M)$.
    The space~$C^{2,1}(M)$ is problematic as described above.
    In order to improve the regularity of~$p$ one needs to make sense of higher regularity and prove a suitable ellipticity result, but we will not explore this avenue.

    \item \label{item:deriv-loss-2} Secondly, we lose a derivative of~$p$ in the transition of regularity from the spherical harmonic components of~$f$ to the spherical harmonic components of the integral function~$u \coloneqq u^f$ of~$f$ (see section~\ref{sec:definitions-notation}).
    Consider the smooth case, where~$g \in C^\infty$, and let~$f = f_m + f_{m-2} + f_{m-4} + \cdots$ and~$u = u_{m-1} + u_{m-3} + u_{m-5} + \cdots$ be the spherical harmonic decompositions of~$f$ and~$u$.
    The geodesic vector field~$X$ on the unit sphere bundle of~$M$ splits into the two operators~$X_+$ and~$X_-$ in each spherical harmonic degree (see section~\ref{sec:definitions-notation}).
    Projecting the transport equation~$Xu = -f$ into each spherical harmonic degree gives~$X_+u_{m-1} = -f_m$ and~$X_+u_{k-1} = -f_k - X_-u_{k+1}$ for~$k \le m-2$ with~$k \equiv m \pmod 2$.
    The operator~$X_+$ is known to be an elliptic pseudodifferential operator of order one (see e.g.~\cite{PSUIDBTTT}) and thus by elliptic regularity we see that each~$u_k$ has one more derivative than the corresponding component~$f_{k+1}$.
    This argument shows that~$u$ has one more derivative than~$f$, proving that~$p$ is~$C^{1,1}$ when~$f$ is Lipschitz.

    However, when~$g \in C^{1,1}(M)$ the phase space~$SM$ is not equipped with a smooth structure and the meaning of ellipticity and its implications such as existence of a parametrix, become less clear.
    The exact formulation and application of ellipticity in the present low regularity setting would be a considerable task and would still not give fully matching regularities in the two parts of theorem~\ref{thm:c11-s-injectivity}.
    Therefore we take a simpler route and do not pursue a fully symmetric version of our main theorem.
\end{enumerate}

%%%%
%%%%
%%%%
\subsection{Related results}

The study of the X-ray transform via the transport equation and Pestov identity approach begun with the work of Mukhometov~\cite{MukhometovIKPSP,MukhometovOPRRM,MukhometovRPTDRM}, where injectivity results for the transform of scalar functions were proved.
Since Mukhometov's seminal articles, the Pestov identity method has been applied to the case of~$1$-forms in~\cite{ARUDFFDIAG} and to higher order tensors in~\cite{PSUIDBTTT,PSUTTS}.
Besides manifolds with boundaries, Pestov identities are useful in the study of integral data of functions and tensor fields over closed curves on closed Anosov manifolds~\cite{CSSRCNCM,DSSPIGAM,PSUSRIDAS,PSUIDBTTT,USDBRSRRSNFP}.
The method is even applicable in non-compact geometries.
For results on Cartan--Hadamard manifolds see~\cite{LehtonenGRTTDCHM,LRSTTCHM}.
There are plenty of other geometrical variations of the problem, which have been studied employing a Pestov identity.
These include reflecting obstacles inside the manifold~\cite{ISBRTRSCO,IPBRTORO}, attenuations and Higgs fields~\cite{SUARTSS,PSUARTCHF,GPSXRTCNC}, manifolds with magnetic flows~\cite{Jol20071,Jol20072,DPSUBRPPMF,AinsworthAMRTS,MPIPGD}, and non-Abelian variations~\cite{FUXRTNACTD,PSNAXRTS,FNPCINNXT,NovikovNARTIT}.
The Pestov identity approach has been studied in more general geometries than Riemannian.
For results in Finsler geometry see~\cite{ADXRTGFCFS,IMGRTSSRFM} and for pseudo-Riemannian geometry~\cite{IlmavirtaXRTPR}.

Only few injectivity results exist outside smooth geometry, whether Riemannian or not.
Injectivity of the scalar X-ray transform is known spherically symmetric~$C^{1,1}$ regular manifolds satisfying the Herglotz condition, when the conformal factor of the metric is~$C^{1,1}$~\cite{HIATLRAXTSSM}.
The scalar (and $1$-form) X-ray transform is (solenoidally) injective on simple~$C^{1,1}$ manifolds~\cite{IKPIXRTMLR}.
The proof of injectivity in~\cite{IKPIXRTMLR} is based on a Pestov identity.

The boundary rigidity problem is a geometrization of the travel time tomography problem and its linearization is the X-ray tomography problem of $2$-tensor fields.
For results in boundary rigidity see~\cite{MukhometovRPTDRMIG,UhlmannIPSU,CrokeRDBBP,CRSNPC,SURMWSLG,MRPFIRMDS,PUTDCSRMBR,GMTBLRNCM,LSUSBRRM,BIBRFVMMCFO}.
For a comprehensive survey on results in travel time tomography and tensor tomography see~\cite{SUVZTTT,IMIGMBA}.

\subsection{Acknowledgements}

Both authors we supported by the Academy of Finland (JI by grant 351665, AK by grant 351656).
AK was supported by the Finnish Academy of Science and
Letters.
This work was supported by the Research Council of Finland (Flagship of Advanced Mathematics for Sensing Imaging and Modelling grant 359208 and Centre of Excellence of Inverse Modelling and Imaging 353092).
We thank the anonymous referees for many valuable comments and suggestions.

%%%%
%%%%
%%%%
\section{Proof of the main theorem}
\label{sec:proof-of-theorem}

%%%%
%%%%
%%%%
\subsection{Basic definitions and notation}
\label{sec:definitions-notation}

In this subsection we present enough terminology and notation to state and prove our main theorem.
The preliminaries of the non-smooth setting are complemented in section~\ref{sec:preliminaries}.

Throughout the article~$M$ will be a compact and connected smooth manifold with a smooth boundary~$\partial M $.
The manifold~$M$ is equipped with a~$C^{1,1}$ regular Riemannian metric~$g$.

\subsubsection{Bundles}
\label{subsubsec:bundles}

The tangent bundle~$TM$ of~$M$ has a subbundle~$SM$ called the unit sphere bundle, which consists of the unit vectors in~$TM$.
As the level set $F^{-1}(1)$ of the~$C^{1,1}$ map~$F \colon TM \to \R$ defined by~$F(x,v) = g_x(v,v)$ the unit sphere bundle is a~$C^{1,1}$ submanifold\footnote{It is easily verified by inspecting the vertical component that the differential $\der F$ is non-zero when $F=1$. The smooth regular level set theorem~\cite{Lee2013} can easily be adapted to our case.} of~$TM$.
The boundary
\begin{equation}
\partial(SM)
\coloneqq
\{\,
(x,v) \in SM
\,:\,
x \in \partial M
\,\}
\end{equation}
of~$SM$ is divided into inwards and outwards pointing parts~$\doo{in}(SM)$ and~$\doo{out}(SM)$ with respect to the inner product~$\ip{\cdot}{\cdot}_{g}$ and a unit normal vector field~$\nu$ to the boundary~$\partial M$.
The subset of $\partial(SM)$ consisting of the vectors~$v$ such that~$\ip{v}{\nu}_g = 0$ is denoted by~$\partial_0(SM)$ and it is disjoint from~$\doo{in}(SM)$ and~$\doo{out}(SM)$.

Let~$\pi \colon SM \to M$ be the standard projection and let~$\pi^\ast(TM)$ be the pullback of~$TM$ over~$SM$.
We denote by~$N$ the subbundle of~$\pi^\ast(TM)$ with the fiber~$N_{(x,v)}$ being the~$g$-orthogonal complement of~$v$ in~$T_xM$.

\subsubsection{Horizontal--vertical decomposition}
\label{subsubsec:hv}

The tangent bundle~$T(SM)$ of~$SM$ has an orthogonal splitting $T(SM) = \R X \oplus \H \oplus \V$ with respect to the so-called Sasaki metric, where~$\H$ and~$\V$ are the horizontal and vertical subbundles respectively and~$X$ is the geodesic vector field on~$SM$.
We denote~$\R X \oplus \H$ by~$\overline{\H}$ and call it the total horizontal subbundle.
Elements of~$\H$ and~$\V$ are respectively referred to as horizontal and vertical derivatives or vectors on~$SM$.
The summands~$\H$ and~$\V$ are each naturally identified with a copy of the bundle~$N$.
The horizontal--vertical geometry is essentially the same as the smooth one (see~\cite{PaternainGF}) and works fine when~$g \in C^{1,1}$.

\subsubsection{Geodesic flow}
\label{subsubsec:flow}

Since the Christoffel symbols of a $C^{1,1}$ metric are Lipschitz, there is a unique unit speed geodesic~$\gamma_z$ corresponding to a given initial condition $z \in SM$ by standard ODE theory.
We define the geodesic flow on the unit sphere bundle to be the collection of (partially defined) maps~$\phi_t \colon SM \to SM$, $\phi_t(z) = (\gamma_z(t),\dot\gamma_z(t))$, where~$t$ goes through all real numbers so that the right-hand side is defined.
The infinitesimal generator~$X$ of the flow is called the geodesic vector field on~$SM$.
For any~$z \in SM$, the geodesic~$\gamma_z$ is defined on a maximal interval of existence~$[-\tau_-(z),\tau_+(z)]$, where~$\tau_-(z)$ and~$\tau_+(z)$ are positive.
We call~$\tau(z) \coloneqq \tau_+(z)$ the travel time function on~$SM$.
The geodesic vector field~$X$ acts naturally on functions by differentiation and on sections~$W$ of the bundle~$N$ it acts by
\begin{equation}
\label{eq:X-action-field}
XW(z) = D_tW(\phi_t(z))|_{t=0},
\end{equation}
where~$D_t$ is the covariant derivative along the curve~$t \mapsto \phi_t(z)$.
The result~$XW$ of the action~\eqref{eq:X-action-field} is again a section of~$N$.

\subsubsection{The X-ray transform}
\label{subsubsec:transform}

Any symmetric~$m$-tensor field~$f$ on~$M$ can be considered as a function on the unit sphere bundle.
Given~$(x,v) \in SM$ we let~$f(x,v) \coloneqq f_x(v,\dots,v)$.
In lemma~\ref{lma:e-ae} and proposition~\ref{prop:identification} and their proofs we denote the induced maps by $\lambda_xf \colon S_xM \to \R$ and $\lambda f \colon SM \to \R$ with $\lambda f(x,v) = \lambda_x f(v)$.
Otherwise we freely identify $f$ with $\lambda f$ since there is no danger of confusion.

The integral function~$u^f \colon SM \to \R$ of a continuous symmetric~$m$-tensor field~$f$ is defined by
\begin{equation}
\label{eq:function-uf}
u^f(x,v)
\coloneqq
\int_0^{\tau(x,v)}
\lambda f(\phi_t(x,v))
\,\d t
\end{equation}
for all~$(x,v) \in SM$.
The X-ray transform of~$f$ is the restriction of the integral function to the inward pointing part of the boundary~$\partial(SM)$, so we may declare~$If \coloneqq u^f|_{\doo{in}(SM)}$.

\subsubsection{Differentiability}

We exclude the rank of the tensor field from our notations for function spaces.
For tensor fields the derivatives are covariant.
We use the subscript~$0$ to indicate zero boundary values.
Thus, for example, $f\in C_0^{1,\alpha}(M)$ for a tensor field~$f$ means that $f|_{\partial M}=0$ and $\nabla f$ is $\alpha$-Hölder.
We use two kinds of functions on the sphere bundle~$SM$, scalars (e.g. $C^1(SM)$) and sections of the bundle~$N$ (e.g. $C^1(N)$) defined in subsection~\ref{subsubsec:bundles}.

We define $\C{k,\alpha}{l,\beta}{SM}$ as the subset of~$C(SM)$ consisting of functions with~$k$ many~$\alpha$-Hölder horizontal derivatives and~$l$ many~$\beta$-Hölder vertical derivatives as well as any combination of~$k$ horizontal and~$l$ vertical derivatives, which are assumed to be~$\omega$-Hölder for~$\omega \coloneqq \min(\alpha,\beta)$.
We let
\begin{equation}
\C{k,\alpha}{\infty}{SM}
\coloneqq
\bigcap_{l=0}^\infty
\C{k,\alpha}{l,1}{SM}.
\end{equation}

According to the splitting~$T(SM) = \R X \oplus \H \oplus \V$, the gradient of a~$C^1$ function~$u$ on~$SM$ can be written as
\begin{equation}
\nabla u
=
((Xu)X,\grad{h}u,\grad{v}u).
\end{equation}
This gives rise to two new differential operators;
the vertical gradient~$\grad{h}$ and the horizontal gradient~$\grad{v}$.
Both~$\grad{h}u$ and~$\grad{v}u$ are naturally identified with sections of the bundle~$N$.
The horizontal and vertical divergences are the~$L^2$ adjoints of the corresponding gradients.
The $L^2$ adjoint of $X$ is $-X$.
The vertical Laplacian on the sphere bundle is $\lap{v} \coloneqq -\dive{v}\grad{v}$; see~\cite[Appendix A]{PSUIDBTTT} for details on the differential operators.

\subsubsection{Curvature}

By Rademacher's theorem a Lipschitz continuous scalar function on a Euclidean domain is differentiable almost everywhere and the derivative is in~$L^\infty$.
Using local coordinates and studying the individual components shows that the Riemann curvature tensor $R_{ijkl}(x)$ corresponding to a Riemannian metric $g\in C^{1,1}$ has all components well defined for almost all $x\in M$.
Thus we may interpret the curvature tensor~$R$ as an~$L^\infty$ tensor field.
The curvature tensor~$R \colon L^\infty(N) \to L^\infty(N)$ acts on sections of the bundle~$N$ by $R(x,v)W(x,v) \coloneqq R(W(x,v),v)v$ producing again~$L^\infty$ sections of the bundle~$N$.

We say that the sectional curvature of the manifold~$M$ is almost everywhere non-positive, if for almost all~$x \in M$ it holds that~$\ip{R(w,v)v}{w}_{g(x)} \le 0$ for all linearly independent~$v,w \in T_xM$.

\subsubsection{Sobolev spaces}
\label{sec:sobolev-spaces}

There are natural~$L^2$ spaces for functions on the sphere bundle as well as for sections of the bundle~$N$, which we will denote by~$L^2(SM)$ and~$L^2(N)$.
We define the Sobolev spaces~$H^1(SM)$ and~$H^1(N,X)$ respectively defined as completions of~$C^{1}(SM)$ and~$C^{1}(N)$ with respect to the norms
\begin{equation}
\begin{split}
\norm{u}^2_{H^1(SM)}
&\coloneqq
\norm{u}^2_{L^2(SM)}
+
\norm{Xu}^2_{L^2(SM)}
+
\norm{\grad{h}u}^2_{L^2(SM)}
+
\norm{\grad{v}u}^2_{L^2(SM)},
\quad\text{and}
\\
\norm{W}^2_{H^1(N,X)}
&\coloneqq
\norm{W}^2_{L^2(N)}
+
\norm{XW}^2_{L^2(N)}.
\end{split}
\end{equation}
We denote zero boundary values by a subindex~$0$. For example, $H_0^1(SM)$ is the subspace of $H^1(SM)$ with zero boundary values.

\subsubsection{Spherical harmonics}

Given~$x \in M$, the unit sphere~$S_xM$ has the Laplace--Beltrami operator~$\lap{v}_x \coloneqq -g^{ij}(x)\partial_{v^i}\partial_{v^j}$.
Letting~$x \in M$ vary we get a second order operator~$\lap{v} = -\dive{v}\grad{v}$ on the unit sphere bundle called the vertical Laplacian, where~$-\dive{v}$ is the formal~$L^2$-adjoint of~$\grad{v}$.

Let~$S^{n-1} \subseteq \R^n$ be the Euclidean unit sphere.
It is well-know that any function~$f \in L^2(S^{n-1})$ can be decomposed as an~$L^2$-convergent series~$f = \sum_{k=0}^\infty f_k$, where~$f_k$ are eigenfunctions of the spherical Laplacian on~$S^{n-1}$ corresponding to the eigenvalues $k(k+n-2)$.
Similarly, any function~$u \in L^2(SM)$ can be decomposed as an~$L^2(SM)$-convergent series~$u = \sum_{k=0}^\infty u_k$,
where $\lap{v}u_k = k(k+n-2)u_k$ for all $k \in \N$. We call $u_k$ the
$k$th spherical harmonic component of $u$. For $k \in \{0,1\}$, $k,l \in \N$ and
$\alpha,\beta \in [0,1]$ we let
\begin{equation}
\OS{k,\alpha}{l,\beta}{m}
\coloneqq
\{\,
u
\in
\C{k,\alpha}{l,\beta}{SM}
\,:\,
\lap{v}u
=
m(m+n-2)u
\,\}
\end{equation}
and
\begin{equation}
\OS{k,\alpha}{\infty}{m}
\coloneqq
\bigcap_{l \in \N}
\OS{k,\alpha}{l,1}{m}.
\end{equation}
Furthermore, we denote
\begin{equation}
\LS{k}{l}{m}
=
\{\,
u
\in
\HS{k}{l}{SM}
\,:\,
\lap{v}u
=
m(m+n-2)u
\,\}.
\end{equation}

For all $m \in \N$ there are operators $X_\pm \colon \OS{1}{\infty}{m} \to \OS{0}{\infty}{m \pm 1}$ with the convention that $\OS{0}{\infty}{-1} = 0$ so that $X = X_+ + X_-$. These mapping properties and validity of this decomposition in low regularity are addressed in proposition~\ref{prop:x-mapping-property}.

\subsubsection{Simple~$C^{1,1}$ manifolds}

The global index form~$Q$ of the manifold~$(M,g)$ (not of a single geodesic) is the quadratic form defined for~$W \in H^1_0(N,X)$ by
\begin{equation}
Q(W)
\coloneqq
\norm{XW}^2_{L^2(N)}
-
\iip{RW}{W}_{L^2(N)}.
\end{equation}
It was proved in~\cite[Lemma 11]{IKPIXRTMLR} that there are no conjugate points on a Riemannian manifold~$(M,g)$, $g \in C^\infty$, if the global index form~$Q$ of~$(M,g)$ is positive definite.

We conclude this subsection by recalling a definition of a simple manifold in the case~$g \in C^{1,1}$.
Our definition is equivalent to the definition of traditional simple manifold when~$g \in C^\infty$~\cite{IKPIXRTMLR}.
Let~$M \subseteq \R^n$ be the closed Euclidean unit ball and let~$g$ be a~$C^{1,1}$ regular Riemannian metric on~$M$.
We say that $(M,g)$ is \emph{a simple~$C^{1,1}$ Riemannian manifold} if the following hold:
\begin{enumerate}[label=A{{\arabic*}}:, ref=A{{\arabic*}}]
\item\label{a1} There is~$\varepsilon > 0$ so that ~$Q(W) \ge \varepsilon\norm{W}^2_{L^2(N)}$ for all~$W \in H_0^1(N,X)$.

\item\label{a2} Any two points of~$M$ can be joined by a unique geodesic in the interior of~$M$, whose length depends continuously on its end points.

\item\label{a3} The squared travel time function~$\tau^2$ (see~\ref{subsubsec:flow}) is Lipschitz on~$SM$.
\end{enumerate}

%%%%
%%%%
%%%%
\subsection{Proof of the theorem}

In this subsection we prove our main result, theorem~\ref{thm:c11-s-injectivity}.
We state the lemmas required for the proof of~\ref{thm:c11-s-injectivity}, and the proofs of the lemmas are postponed to sections~\ref{sec:bnd-determ-regularity}, \ref{sec:estimates-santalo}, and~\ref{sec:friedrichs}.

\begin{lemma}[Boundary determination]
\label{lma:boundary-determination}
Let $(M,g)$ be a simple~$C^{1,1}$ manifold. If $f \in C^{1,1}(M)$ is a symmetric~$m$-tensor field with~$If = 0$, then there is a symmetric~$(m-1)$-tensor field~$p \in C^{1,1}(M)$ so that $f|_{\partial M} = \sigma\nabla p|_{\partial M}$ and~$p|_{\partial M} = 0$.
\end{lemma}

\begin{lemma}[Regularity of spherical harmonic components]
\label{lma:regularity-u}
Let $(M,g)$ be a simple~$C^{1,1}$ manifold.
Let $f \in \Lip_0(M)$ be a symmetric~$m$-tensor field on~$M$ with~$If=0$ and let~$u \coloneqq u^f$ be the integral function of~$f$ defined by~\eqref{eq:function-uf}.
If the spherical harmonic decomposition of~$u$ is~$u = \sum_{k=0}^\infty u_k$, then~$u_k \in \OS{0,1}{\infty}{k}$ and~$u_k|_{\partial(SM)} = 0$ for all $k \in \N$.
\end{lemma}

\begin{lemma}
\label{lma:x+uL2}
Let $(M,g)$ be a simple~$C^{1,1}$ manifold.
Let $f \in \Lip_0(M)$ be a symmetric~$m$-tensor field on~$M$ with~$If=0$ and let~$u \coloneqq u^f$ be the integral function of~$f$ defined by~\eqref{eq:function-uf}.
Then $X_+u\in L^2(SM)$.
\end{lemma}

Lemma~\ref{lma:x+uL2} follows immediately from lemmas~\ref{lma:regularity-u} and~\ref{lma:xpm-convergence}.

Recall that~$n$ is the dimension of~$M$.
For natural numbers~$k$ and~$l$ we define the two constants
\begin{equation}
C(n,k)
\coloneqq
\frac{2k+n-1}{2k+n-3}
\quad\text{and}\quad
B(n,l,k)
\coloneqq
\prod^l_{p=1}
C(n,k+2p).
\end{equation}

\begin{lemma}
\label{lma:l2-estimate}
Let~$(M,g)$ be a simple~$C^{1,1}$ manifold with almost everywhere non-positive sectional curvature.
Let~$f \in \Lip_0(M)$ be a symmetric~$m$-tensor field with~$If=0$ and denote by~$u \coloneqq u^f$ the integral function of~$f$ defined by~\eqref{eq:function-uf}.
If the spherical harmonic decomposition of~$u$ is~$u = \sum_{k = 0}^\infty u_k$, then for all $k \ge m$ and $l \in \N$ we have
\begin{equation}
\norm{X_+u_{k}}^2_{L^2(SM)}
\le
B(n,l,k)
\norm{X_+u_{k+2l}}^2_{L^2(SM)}.
\end{equation}
\end{lemma}

\begin{lemma}[Injectivity of~$X_+$]
\label{lma:injectivity-xplus}
Let~$(M,g)$ be a simple~$C^{1,1}$ manifold with almost everywhere non-positive sectional curvature.
Suppose that $u \in \OS{0,1}{\infty}{k}$ and~$u|_{\partial(SM)} = 0$. Then $X_+u = 0$ implies that~$u = 0$.
\end{lemma}

\begin{lemma}
\label{lma:e-ae}
Let~$(M,g)$ be a simple~$C^{1,1}$ manifold.
Let~$f \in \Lip(M)$ be a symmetric~$m$-tensor field.
Suppose that~$p$ is a symmetric~$(m-1)$-tensor field and~$u = -\lambda p$ is a Lipschitz function in~$SM$ so that~$Xu = -\lambda f$ everywhere in~$SM$.
Then~$\sigma \nabla p = f$ almost everywhere in~$M$.
\end{lemma}

\begin{proof}[Proof of theorem~\ref{thm:c11-s-injectivity}]
Item~\ref{item:thm1}: Suppose that~$p \in C^{1,1}(M)$ is a symmetric~$(m-1)$-tensor field vanishing on~$\partial(M)$. Then using the fundamental theorem of calculus along each geodesic~$If = I(\sigma\nabla p) = 0$ (see~\cite[Lemma 6.4.2]{PSUGIPETD}), which proves item~\ref{item:thm1}.

Item~\ref{item:thm2}: Suppose that the X-ray transform of a symmetric~$m$-tensor field~$f \in C^{1,1}(M)$ vanishes.
We will prove that there is a symmetric~$(m-1)$-tensor field~$p$ vanishing on~$\partial M$ so that~$f = \sigma\nabla p$.

By boundary determination in lemma~\ref{lma:boundary-determination} there is a symmetric~$(m-1)$-tensor field~$p_0 \in C^{1,1}(M)$ so that~$p_0|_{\partial M} = 0$ and~$f|_{\partial M} = \sigma\nabla p_0|_{\partial M}$.
Let~$\hat f \coloneqq f - \sigma\nabla p_0$.
Then~$\hat f \in \Lip_0(M)$ is a symmetric $m$-tensor field on~$M$ and~$I\hat f = If = 0$.

Let~$u = \sum_{k=0}^\infty u_k$ be the spherical harmonic decomposition of~$u\coloneqq u^{\hat f}$.
Then~$u_k \in \OS{0,1}{\infty}{k}$ by lemma~\ref{lma:regularity-u}. First, we prove that~$u_k = 0$ for all~$k$ for which~$k \equiv m \pmod 2$.

Since for all~$(x,v) \in SM$ it holds that $\hat f(x,-v) = (-1)^m\hat f(x,v)$, we have
\begin{equation}
\begin{split}
u(x,-v)
&=
\int_0^{\tau_+(x,-v)}
\hat f(\gamma_{x,-v}(t),\dot\gamma_{x,-v}(t))
\,\d t
\\
&=
(-1)^m\int_{-\tau_-(x,v)}^0
\hat f(\gamma_{x,v}(t),\dot\gamma_{x,v}(t))
\,\d t.
\end{split}
\end{equation}
Therefore
\begin{equation}
u(x,-v)
+
(-1)^mu(x,v)
=
(-1)^mI\hat f(\phi_{-\tau_-(x,v)}(x,v))
=
0.
\end{equation}
This shows that~$u(x,-v) = (-1)^{m+1}u(x,v)$ for all~$(x,v) \in SM$ and thus~$u_k = 0$ whenever~$k \equiv m \pmod 2$. Next, we will show that~$u_k = 0$ for all~$k \ge m$.

Let~$m_0 \ge m$ and suppose that~$A_1 \coloneqq \norm{X_+u_{m_0}}^2_{L^2(SM)} > 0$.
For all~$l \in \N$ lemma~\ref{lma:l2-estimate} yields the estimate
\begin{equation}
\label{eq:l2-estimate}
B(n,l,m_0)^{-1}
\norm{X_+u_{m_0}}^2_{L^2(SM)}
\le
\norm{X_+u_{m_0 + 2l}}^2_{L^2(SM)}.
\end{equation}
By an elementary estimate (see~\cite[Lemma 13]{IPBRTORO}) there is a constant~$A_2 > 0$ only depending on~$m_0$ and~$n$ so that
\begin{equation}
B(n,l,m_0)^{-1}
\ge
\left(
1 + \frac{4l}{2m_0 + n - 3}
\right)^{-1/2}
\ge
A_2l^{-1/2}.
\end{equation}
Thus the estimate~\eqref{eq:l2-estimate} gives
\begin{equation}
\sum_{l = 1}^\infty
\norm{X_+u_{m_0 + 2l}}^2_{L^2(SM)}
\ge
A_1A_2
\sum_{l = 1}^\infty
l^{-1/2}
=\infty.
\end{equation}
On the other hand $X_+u \in L^2(SM)$ by lemma~\ref{lma:x+uL2}.
Hence orthogonality implies that
\begin{equation}
\sum_{l = 1}^\infty
\norm{X_+u_{m_0 + 2l}}^2_{L^2(SM)}
\le
\sum_{k = 0}^\infty
\norm{X_+u_{k}}^2_{L^2(SM)}
\le
\norm{X_+u}^2_{L^2(SM)}
<
\infty.
\end{equation}
This contradiction proves that~$\norm{X_+u_k}^2_{L^2(SM)} = 0$ for all~$k \ge m$. Since additionally $u_k|_{\partial(SM)} = 0$ for~$k \ge m$, lemma~\ref{lma:injectivity-xplus} says~$u_k = 0$ for all~$k \ge m$.

We have shown~$u_k = 0$ for $k \ge m$ and~$u_k = 0$ for~$k \equiv m \pmod 2$.
Thus~$-u \in \Lip_0(SM)$ is identified with a symmetric~$(m-1)$-tensor field~$p_1 \in \Lip_0(M)$.
As~$u$ solves the transport equation~$Xu = -\hat f$ everywhere on~$SM$ we have~$\sigma\nabla p_1 = \hat f$ almost everywhere on~$M$ by lemma~\ref{lma:e-ae}.
Thus we conclude that~$f = \sigma\nabla p$ almost everywhere in~$M$, where $p \coloneqq p_0 + p_1 \in \Lip(M)$ is a symmetric~$(m-1)$-tensor field with~$p|_{\partial M} = 0$.
\end{proof}

%%%%
%%%%
%%%%
\section{Preliminaries}
\label{sec:preliminaries}

In this article we consider compact and connected smooth manifolds with smooth boundaries.
We assume that such a manifold~$M$ comes equipped with a symmetric and positive definite~$2$-tensor field~$g$ so that its component functions~$g_{jk}$ are~$C^{1,1}$-functions on~$M$.
In this case we refer to~$g$ as a~$C^{1,1}$ \emph{Riemannian metric} and to~$(M,g)$ as a \emph{(non-smooth) Riemannian manifold}.

%%%%
%%%%
%%%%
\subsection{Spaces of tensor fields}
\label{subsec:tensors}

Since~$g$ is a~$C^{1,1}$ Riemannian metric, componentwise differentiability and existence of covariant derivatives are not the same.
Even if the components of a tensor field~$f$ in any local coordinates are~$C^k$ functions for~$k \ge 2$ (which is possible since~$M$ is assumed to have a smooth structure), the covariant derivative~$\nabla f$ falls into~$\Lip(M)$.
Since most of our considerations are related to the metric structure and componentwise differentiability is not compatible with the covariant derivative, the correct definition of a~$C^{1,1}$ tensor field is by covariant differentiability.
However, with covariant differentiability we are restricted to~$C^{1,1}(M)$ and higher regularity does not exist on the Hölder scale.

The space~$L^2(M)$ of~$L^2$-tensor fields of order~$m$ on~$M$ is defined to be the completion of the space of continuous~$m$-tensor fields with respect to the norm induced by the inner product
\begin{equation}
\iip{f}{h}_{L^2(M)}
\coloneqq
\int_M
g^{j_1k_1}\cdots g^{j_mk_m}
f_{j_1\cdots j_m}
h_{k_1\cdots k_m}
\,\d V_g.
\end{equation}
Here~$\d V_g$ is the Riemannian volume form of~$M$.
The space~$H^1(M)$ of~$H^1$-tensor fields of order~$m$ on~$M$ is defined to be the closure of the space of continuously differentiable~$m$-tensor fields with respect to the norm
\begin{equation}
\norm{f}^2_{H^1(M)}
\coloneqq
\norm{f}^2_{L^2(M)}
+
\norm{\nabla f}^2_{L^2(M)}.
\end{equation}

Let~$p \in [1,\infty)$.
The spaces~$L^p(M)$ and~$W^{1,p}(M)$ of~$L^p$- and~$W^{1,p}$-tensor fields of order~$m$ are defined analogously to the spaces~$L^2(M)$ and~$H^1(M)$.

We could give definitions of the spaces~$H^2(M)$ and~$W^{2,p}(M)$ for tensor fields of any order similar to the definitions of spaces~$H^1(M)$ and~$W^{1,p}(M)$.
Again, since~$g$ is only a~$C^{1,1}$ regular Riemannian metric, there are no spaces~$H^3(M)$ and~$W^{3,p}(M)$ compatible with the geometry.
A compatible space should be defined using covariant derivatives in the norms, which would force the spaces~$W^{k,p}(M)$ trivial, when~$k \ge 3$.

If~$f \in C^1(M)$ is a symmetric~$m$-tensor field on~$M$, its symmetrized covariant derivative is~$\sigma\nabla f$.
The symmetrization~$\sigma$ is defined for all~$m$-tensor fields~$h$ on~$M$ by
\begin{equation}
(\sigma h)_{j_1\cdots j_m}
\coloneqq
\frac{1}{m!}
\sum_{\pi}
h_{j_{\pi(1)}\cdots j_{\pi(m)}}
\end{equation}
where the summation is over all permutations~$\pi$ of~$\{1,\dots,m\}$. Note that since $\norm{\sigma \nabla f}_{L^2} \le \norm{\nabla f}_{L^2}$ the symmetrized covariant derivative is bounded between Sobolev spaces.

The trace of a symmetric~$m$-tensor field~$f$ on~$M$ is denoted by~$\trace_g(f)$.
In local coordinates~$\trace_g(f)_{i_1\cdots i_{m-2}}= g^{jk}f_{jki_1\cdots i_{m-2}}$.
A symmetric~$m$-tensor field is called trace-free, if its trace is zero.

%%%%
%%%%
%%%%
\subsection{Vertical and horizontal differentiability}

Let~$M$ be a compact smooth manifold with a smooth boundary and let~$g$ be a~$C^{1,1}$ Riemannian metric on~$M$.
Let~$k \in \N$ and~$\alpha \in [0,1]$ be so that~$k + \alpha \le 2$. For~$l \in \N$ and~$\beta \in [0,1]$ the set~$\C{k,\alpha}{l,\beta}{SM}$ consists of all functions~$u \in C(SM)$ with
\begin{equation}
H_1\cdots H_ku \in C^{0,\alpha}(SM)
\quad\text{and}\quad
V_1\cdots V_lu \in C^{0,\beta}(SM)
\end{equation}
for any~$k$ vector fields $H_1,\dots,H_k \in \overline{\H}$ and any~$l$ vector fields $V_1,\dots,V_l \in \V$.
Additionally, we require that for any~$k+l$ vector fields $Z_1,\dots,Z_{k+l} \in T(SM)$ out of which exactly~$k$ are in~$\overline{\H}$ and exactly~$l$ are in~$\V$ we have
\begin{equation}
Z_1\cdots Z_{k+l}u \in C^{0,\omega}(SM),
\quad\text{where}\quad
\omega \coloneqq \min(\alpha,\beta).
\end{equation}
We let
\begin{equation}
\C{k,\alpha}{\infty}{SM}
\coloneqq
\bigcap_{l \in \N}
\C{k,\alpha}{l,1}{SM}.
\end{equation}

\begin{remark}
In the definition of~$\C{k,\alpha}{l,\beta}{SM}$ the vertical differentiability indices~$l$ and~$\beta$ can surpass the smoothness of charts of~$SM$.
It is not necessary for~$SM$ to have~$C^\infty$ smooth charts, since vertical vector fields operate on a fixed fiber and for a fixed point~$x$ in~$M$ the scaling~$s(x,v) = (x,v\abs{v}^{-1}_g)$ is smooth on~$T_xM \setminus 0$.
The slit tangent space~$T_xM\setminus0$ has a smooth structure even if~$M$ does not.
\end{remark}

\begin{remark}
Any commutator~$[H,V] = HV - VH$, where~$H \in \overline{\H}$ and~$V \in \V$, can be defined classically on the space~$\C{1}{1}{SM}$, since for any $u \in \C{1}{1}{SM}$ the derivatives~$HVu$ and $VHu$ are in~$C(SM)$.
\end{remark}

The set~$\C{k,\alpha}{l,\beta}{N}$ consists of all continuous sections~$W$ of the bundle~$N$ with~$W^j$ in~$\C{k,\alpha}{l,\beta}{SM}$ when~$W = W^j\partial_{x^j}$.
A section~$W$ of the bundle~$N$ is continuous, if it is continuous as a map~$SM \to TM$.

As one might expect, vertical operators preserve horizontal differentiability and horizontal operators preserve vertical differentiability. That is
\begin{align}
X &\colon 
\C{k,\alpha}{l,\beta}{SM}
\to
\C{k-1,\alpha}{l,\beta}{SM},
\\
X &\colon 
\C{k,\alpha}{l,\beta}{N}
\to
\C{k-1,\alpha}{l,\beta}{N},
\\
\grad{v} &\colon
\C{k,\alpha}{l,\beta}{SM}
\to
\C{k,\alpha}{l-1,\beta}{N},
\\
\dive{v} &\colon
\C{k,\alpha}{l,\beta}{N}
\to
\C{k,\alpha}{l-1,\beta}{SM},
\\
\grad{h} &\colon
\C{k,\alpha}{l,\beta}{SM}
\to
\C{k-1,\alpha}{l,\beta}{N}, \quad\text{and}
\\
\dive{h} &\colon
\C{k,\alpha}{l,\beta}{N}
\to
\C{k-1,\alpha}{l,\beta}{SM}.
\end{align}

%%%%
%%%%
%%%%
\subsection{Sobolev spaces of different vertical and horizontal indices}

Standard Sobolev spaces on~$SM$ are defined in section~\ref{sec:sobolev-spaces}. Here, we define Sobolev spaces for scalar functions on~$SM$ of different vertical and horizontal indices.
If~$k,l \in \{0,1\}$ and~$u$ is a scalar function in $\C{k}{l}{SM}$ we define the $\HS{k}{l}{SM}$-norm of~$u$ to be
\begin{equation}
\norm{u}^2_{\HS{k}{l}{SM}}
\coloneqq
\norm{u}^2_{L^2(SM)}
+
k\norm{Xu}^2_{L^2(SM)}
+
k\norm{\grad{h}u}^2_{L^2(N)}
+
l\norm{\grad{v}u}^2_{L^2(N)}.
\end{equation}
The Sobolev space~$\HS{k}{l}{SM}$ for~$k,l \in \{0,1\}$ is defined to be the completion of~$\C{k}{l}{SM}$ with respect to the norm~$\norm{\cdot}_{\HS{k}{l}{SM}}$.

Similarly, we define spaces~$\HS{0}{2}{SM}$ and $\HS{1}{2}{SM}$ to be the completions of~$\C{0}{2}{SM}$ and of~$\C{1}{2}{SM}$ with respect to the norms
\begin{align}
\norm{u}^2_{\HS{0}{2}{SM}}
&\coloneqq
\norm{u}^2_{L^2(SM)}
+
\norm{\lap{v} u}^2_{L^2(SM)},
\quad\text{and}
\\
\norm{u}^2_{\HS{1}{2}{SM}}
&\coloneqq
\norm{u}^2_{\HS{1}{1}{SM}}
+
\norm{u}^2_{\HS{0}{2}{SM}}
\\
&\hspace{5mm}+
\norm{X\lap{v} u}^2_{L^2(SM)}
+
\norm{\lap{v} X u}^2_{L^2(SM)}.
\end{align}
Note that the norm on~$\HS{1}{2}{SM}$ does not cover all possible combinations of a horizontal derivative and two vertical derivatives (e.g.~$\dive{v}X\grad{v}$). This is intentional, since the missing combinations will not be needed.

\begin{proposition}
\label{prop:commutator-formulas}
Let~$M$ be a compact smooth manifold with a smooth boundary and let~$g$ be a~$C^{1,1}$ Riemannian metric on~$M$.
The following commutator formulas hold on~$\HS{1}{2}{SM}$:
\begin{align}
\label{eq:comm-x-gradv}
[X,\grad{v}]
&=
-\grad{h},
\\
\label{eq:comm-div}
\dive{h}\grad{v}
-
\dive{v}\grad{h}
&=
(n-1)X,
\\
\label{eq:comm-laplace-x}
[X,\lap{v}]
&=
2\dive{v}\grad{h}
+
(n-1)X.
\end{align}
The following commutator formula holds on~$\HS{1}{1}{N}$:
\begin{equation}
\label{eq:comm-x-divev}
[X,\dive{v}]
=
-\dive{h}.
\end{equation}
\end{proposition}

\begin{proof}
Formulas~\eqref{eq:comm-x-gradv},~\eqref{eq:comm-div} and~\eqref{eq:comm-laplace-x} on~$\C{1}{2}{SM}$ and~\eqref{eq:comm-x-divev} on~$\C{1}{1}{N}$ can be proved by a computation similar to~\cite[Appendix]{PSUIDBTTT}, since the computations use one horizontal derivative and two vertical for~\eqref{eq:comm-x-gradv},~\eqref{eq:comm-div} and~\eqref{eq:comm-laplace-x} and one horizontal and one vertical derivative for~\eqref{eq:comm-x-divev}. The same formulas hold on~$\HS{1}{2}{SM}$ and~$\HS{1}{1}{N}$ by approximation.
\end{proof}

%%%%
%%%%
%%%%
\subsection{Vertical Fourier analysis}

In this subsection we recall the identification of trace-free symmetric tensor fields and spherical harmonics (the vertical Fourier modes).
We state and prove proposition~\ref{prop:identification} in order to emphasize what changes in these well known results when applied to a case of non-smooth Riemannian metrics.
More details in the case of~$C^\infty$-smooth Riemannian metrics can be found for example in~\cite{PSUGIPETD} and~\cite{DSCKSTFRM}.

\begin{proposition}
\label{prop:identification}
Let~$M$ be a compact smooth manifold with a smooth boundary and let~$g$ be a~$C^{1,1}$ Riemannian metric on~$M$. Let~$k \in \{0,1\}$ and~$\alpha \in [0,1]$. The map~$\lambda \colon f \mapsto \lambda f$ is defines a linear isomorphism from the space of symmetric trace-free~$m$-tensor fields in~$C^{k,\alpha}(M)$ to the space~$\OS{k,\alpha}{\infty}{m}$. There is a constant~$C_{m,n} > 0$ so that for all symmetric trace-free~$m$-tensor fields~$f \in C^0(M)$ we have
\begin{equation}
\label{eq:norm-preserved}
\norm{\lambda f}_{L^2(SM)}
=
C_{m,n}\norm{f}_{L^2(M)}.
\end{equation}
Furthermore, there are positive constants~$c,C > 0$ so that for any two~$m$-tensor fields~$f$ and~$h$ in~$C^0(M)$ we have
\begin{equation}
\label{eq:inner-product-comparable}
c\iip{\lambda f}{\lambda h}_{L^2(SM)}
\le
\iip{f}{h}_{L^2(M)}
\le
C\iip{\lambda f}{\lambda h}_{L^2(SM)}.
\end{equation}
\end{proposition}

\begin{proof}
As in the smooth case~\cite[Lemma 2.5.]{DSCKSTFRM} the map~$\lambda_x$ isomophically maps trace-free~$m$-tensors to spherical harmonics~$S_xM$ of degree~$m$. Since the dependence on~$x$ is of the form~$\lambda f(x,v) = f_{j_1\dots j_m}(x)v^{j_1}\cdots v^{j_m}$, the map~$\lambda$ maps on trace-free $m$-tensor fields in~$C^{k,\alpha}(M)$ into~$\OS{k,\alpha}{\infty}{m}$.

For any symmetric and trace-free $m$-tensor fields~$f,h\in C^0(M)$, a fiberwise calculation \cite[Lemma 2.4.]{DSCKSTFRM} shows that for all~$x \in M$ we have
\begin{equation}
\label{eq:inner-prod-preserved}
\int_{S_xM}
(\lambda_xf)(\lambda_xh)
\,\d S_x
=
C_{m,n}
\ip{f}{h}_{g(x)}
\end{equation}
for some~$C_{m,n} > 0$. Since the computation is fiberwise, it remains valid when~$g \in C^{1,1}$.
Integrating equation~\eqref{eq:inner-prod-preserved} over~$M$ gives
\begin{equation}
\iip{\lambda f}{\lambda h}_{L^2(SM)}
=
C_{m,n}
\iip{f}{h}_{L^2(M)},
\end{equation}
which proves~\eqref{eq:norm-preserved}.
Furthermore, the last claim~\eqref{eq:inner-product-comparable} follows from~\eqref{eq:inner-prod-preserved}, since any symmetric~$m$-tensor field can be decomposed into a sum of symmetric trace-free tensor fields of orders less than or equal to~$m$~\cite{PSUGIPETD}.
\end{proof}

%%%%
%%%%
%%%%
\subsection{Decomposition of the geodesic vector field}

In this subsection we recall the fact that the geodesic vector field maps from spherical harmonic degree~$m$ to spherical harmonic degrees~$m-1$ and~$m+1$.
This mapping property induces a decomposition of~$X$ into operators~$X_+$ and~$X_-$. See~\cite[Section 6.6.]{PSUGIPETD} for details of the decomposition when~$g \in C^\infty$.
We record in proposition~\ref{prop:x-mapping-property} what changes in the decomposition, when the Riemannian metric~$g$ is only~$C^{1,1}$-smooth.

\begin{proposition}
\label{prop:x-mapping-property}
Let~$M$ be a compact smooth manifold with a smooth boundary and let~$g$ be a~$C^{1,1}$ Riemannian metric on~$M$. The geodesic vector field maps
\begin{equation}
X\colon
\OS{1}{\infty}{m}
\to
\OS{0}{\infty}{m-1}
\oplus
\OS{0}{\infty}{m+1}.
\end{equation}
Therefore~$X$ decomposes into operators~$X_+$ and~$X_-$ in each spherical harmonic degree so that
\begin{equation}
X_\pm\colon
\OS{1}{\infty}{m}
\to
\OS{0}{\infty}{m\pm 1}.
\end{equation}
\end{proposition}

\begin{proof}
Let~$u \in \OS{1}{\infty}{m}$ and pick a point~$x \in M$.
Then~$Xu(x,v) = v^j\delta_ju(x,v)$ for all~$v \in S_xM$, where~$v^j$ is a spherical harmonic of degree~$1$ on~$S_xM$ and~$\delta_ju(x,\cdot)$ is a spherical harmonic of degree~$m$ on~$S_xM$.
Since any product of spherical harmonics of degrees~$1$ and~$m$ is a sum of spherical harmonics of degrees~$m-1$ and~$m+1$ we see that
\begin{equation}
X \colon
\OS{1}{\infty}{m}
\to
\OS{0}{\infty}{m-1}
\oplus
\OS{0}{\infty}{m+1}.
\end{equation}
Here the spherical harmonic components of~$Xu$ have one horizontal derivative less than~$u$ since~$X \in \overline{\H}$.
\end{proof}

\begin{remark}
Since~$X$ maps continuously with respect to the~$H^1$- and~$L^2$-norms the mapping properties from proposition~\ref{prop:x-mapping-property} carry over to the Sobolev space.
In other words
\begin{equation}
\begin{split}
X&\colon
\LS{1}{2}{m}
\to
\LS{0}{2}{m-1}
\oplus
\LS{0}{2}{m+1},
\quad\text{and}
\\
X_\pm&\colon
\LS{1}{2}{m}
\to
\LS{0}{2}{m\pm 1}.
\end{split}
\end{equation}
\end{remark}

As stated above, proposition~\ref{prop:x-mapping-property} gives degreewise defined operators~$X_-$ and~$X_+$ acting on~$\LS{1}{2}{SM}$.
If~$u \in \HS{1}{2}{SM}$ and~$u = \sum_{k=0}^\infty u_k$ is the spherical harmonic decomposition of~$u$, we define
\begin{equation}
\label{eq:xpm}
X_\pm u
=
\sum_{k=0}^\infty
X_\pm u_k.
\end{equation}
We prove in lemma~\ref{lma:xpm-convergence} that the series in~\eqref{eq:xpm} converges (absolutely) in~$L^2(SM)$.

The following lemma~\ref{lma:commutators2} is a low regularity version of~\cite[Lemma 3.3.]{PSUIDBTTT}, the only difference being the regularity of~$u$.

\begin{lemma}
\label{lma:commutators2}
Let~$M$ be a compact smooth manifold with a smooth boundary and let~$g$ be a~$C^{1,1}$ Riemannian metric on~$M$.
If~$u \in \LS{1}{2}{m}$ then
\begin{align}
[X_+,\lap{v}]u
&=
-(2m+n-1)X_+u,
\quad\text{and}
\\
\label{eq:commutators2}
[X_-,\lap{v}]u
&=
(2m+n-3)X_-u.
\end{align}
\end{lemma}

\begin{proof}
By density it is enough to prove the claimed formulas for~$u \in \OS{1}{\infty}{m}$.
By eigenvalue property of~$u$ and by the mapping property of~$X_+$ we have
\begin{equation}
\label{eq:comm-laplace-xpm1}
X_+\lap{v} u
=
m(m+n-2)X_+u.
\end{equation}
Similarly, by the eigenvalue property of~$X_+u$ we have
\begin{equation}
\label{eq:comm-laplace-xpm2}
\lap{v} X_+u
=
(m+1)((m+1)+n-2)X_+u.
\end{equation}
Subtracting~\eqref{eq:comm-laplace-xpm1} from~\eqref{eq:comm-laplace-xpm2} shows that
\begin{equation}
[X_+,\lap{v}]u
=
-(2m+n-1)X_+u.
\end{equation}
The identity~\eqref{eq:commutators2} can be proved similarly.
\end{proof}

%%%%
%%%%
%%%%
\section{Boundary determination and regularity lemmas}
\label{sec:bnd-determ-regularity}

This section is devoted to the study of the integral function~$u^f$ of a tensor field~$f$ with vanishing X-ray transform.
We prove a vital boundary determination result (lemma~\ref{lma:boundary-determination}) that allows us to prove that~$u^f$ is a Lipschitz function on~$SM$ in subsection~\ref{subsec:regularity-u}.
In subsection~\ref{subsec:reg-of-sp-harmonics} we exploit the particular form of the identification of trace-free tensor fields and spherical harmonics to prove our main regularity lemma~\ref{lma:regularity-u}.

%%%%
%%%%
%%%%
\subsection{Boundary determination}

The boundary determination lemma~\ref{lma:boundary-determination} is proved in two parts.
In lemma~\ref{lma:local-bnd-determination} we give an explicit local construction.
In more detail, we prove that if~$If$ vanishes for some tensor field~$f$, then in local coordinates near any boundary point we construct a tensor field~$p$ so that the symmetrized covariant derivative of~$p$ equals~$f$ when restricted to the boundary.
We prove that lemma~\ref{lma:boundary-determination} follows from the local construction by a partition of unity argument.

\begin{lemma}
\label{lma:local-bnd-determination}
Let~$(M,g)$ be a simple~$C^{1,1}$ manifold and suppose that~$f \in C^{1,1}(M)$ is a symmetric~$m$-tensor field on~$M$ so that in~$If = 0$.
For each~$x \in \partial M$ there is a neighbourhood~$W \subseteq M$ of~$x$ and a symmetric~$(m-1)$-tensor field~$p \in C^{1,1}(W)$ so that~$p|_{W \cap \partial M} = 0$ and~$\sigma\nabla p|_{W \cap \partial M} = f|_{W \cap \partial M}$.
\end{lemma}

\begin{proof}
Let~$x_0 \in \partial M$ be a boundary point. Choose a neighbourhood~$W \subseteq M$ of~$x_0$, where we have~$C^\infty$ coordinates~$\phi \colon W  \to \R^n$ so that
\begin{equation}
\phi(W \cap \partial M)
=
\{x^n = 0\}
\quad\text{and}\quad
\phi(W \cap M^{\sisus{}})
=
\{x^n > 0\}.
\end{equation}
The smooth coordinate function~$\phi$ exists, since~$M$ is a smooth manifold with a smooth boundary.
Denote $\hat x \coloneqq (x^1,\dots,x^{n-1})$ so that~$x = (\hat x,x^n)$.

In these coordinates the required tensor field~$p$ can be defined in the following way. Given~$l \in \{0,\dots,m-1\}$ and $j_1,\dots,j_l \in \{1,\dots,n-1\}$ we let the component of~$p$ corresponding to the indices~$j_1\cdots j_ln \cdots n$ be
\begin{equation}
p_{j_1\cdots j_ln \cdots n}(\hat x,x^n)
\coloneqq
\frac{m}{m-l}
x^nf_{j_1\cdots j_ln \cdots n}(\hat x,0).
\end{equation}
Here the index~$n$ appears~$m-1-l$ times in~$p_{j_1\cdots j_ln \cdots n}$ and~$m-l$ times in~$f_{j_1\cdots j_ln \cdots n}$.
We can insist that~$p$ is symmetric by requiring
\begin{equation}
p_{j_1\cdots j_{m-1}}(\hat x,x^n)
=
p_{j_{\pi(1)}\cdots j_{\pi(m-1)}}(\hat x,x^n),
\end{equation}
where~$\pi$ is any permutation of~$\{1,\dots,m-1\}$ so that $j_{\pi(1)} \le \dots \le j_{\pi(m-1)}$.
This causes no contradictions, since~$f$ is symmetric.
Clearly, it holds that~$p|_{x^n = 0} = 0$ and~$p \in C^{1,1}(M)$ since~$f \in C^{1,1}(M)$.

It remains to show that~$\sigma\nabla p|_{x^n = 0} = f|_{x^n = 0}$, which follows from two claims:
\begin{enumerate}
    \item \label{item:constr-p-1} We prove $f_{j_1\cdots j_m}(\hat x,0) = 0$ in the coordinates in~$W$ when~$j_1,\dots,j_m \in \{1,\dots,n-1\}$.
    
    \item \label{item:constr-p-2} We verify that~$(\sigma\nabla p)_{j_1\dots j_m}|_{x^n = 0} = f_{j_1\dots j_m}|_{x^n = 0}$ in the coordinates in~$W$.
\end{enumerate}
Both claims are proved in appendix~\ref{app:construction-potential}.
The idea is that item~\ref{item:constr-p-1} follows from the fact~$If = 0$, and item~\ref{item:constr-p-2} can then be verified by a straightforward computation in the coordinates in~$W$.
\end{proof}

\begin{proof}[Proof of lemma~\ref{lma:boundary-determination}]
Let~$f \in C^{1,1}(M)$ be a symmetric~$m$-tensor field with~$If = 0$.
We construct a symmetric~$(m-1)$-tensor field~$p \in C^{1,1}(M)$ so that~$p|_{\partial M} = 0$ and~$\sigma\nabla p|_{\partial M} = f|_{\partial M}$.

For each~$x \in \partial M$ pick a neighbourhood~$W_x \subseteq M$ of~$x$ and a symmetric~$(m-1)$-tensor field~$p_x \in C^{1,1}(W_x)$.
Such neighbourhoods~$W_x$ and tensor fields~$p_x$ exist by lemma~\ref{lma:local-bnd-determination}.
Since~$\partial M$ is compact, there is a finite subcover~$\{W_{x_i}\}_{i=1}^k$ of the open cover~$\{W_x\}_{x \in \partial M}$ of~$\partial M$.
Denote $W_i \coloneqq W_{x_i}$ and~$p_i \coloneqq p_{x_i}$.
We add~$W_0 \coloneqq M \setminus \partial M$ to get a finite open cover of~$M$.
Choose a partition of unity~$\{\psi_i\}_{i=1}^n \cup \{\psi_0\}$ subordinate to $\{W_i\}_{i=1}^n \cup \{W_0\}$.
We let the tensor field~$p_0$ corresponding to~$W_0$ to be identically zero.
The products~$\psi_ip_i$ are~$C^{1,1}$ tensor fields in neighbourhoods~$W_i$ and we can extend them by zero outside~$W_i$ to get~$C^{1,1}$ tensor fields on~$M$ since each~$W_i \setminus \supp \psi_i$ is open. 
We define an~$(m-1)$-tensor field~$p$ by
\begin{equation}
p(x) = \sum_{i = 0}^n\psi_i(x)p_i(x).
\end{equation}
Since $\psi_ip_i$ are zero outside~$\supp \psi_i$ and~$p_i|_{\partial M \cap \supp \psi_i} = 0$ by construction, we see that~$p|_{\partial M} = 0$.
The final step is to check that~$\sigma\nabla p = f$ on the boundary~$\partial M$.
By the product rule we have~$\nabla(\psi_ip_i) = \nabla \psi_i \otimes p_i + \psi_i(\nabla p_i)$ for all~$i$.
Since symmetrization commutes with multiplication by a scalar function and $\psi_i$ is a scalar we have
\begin{equation}
\sigma\nabla p
=
\sum_{i = 0}^n
[
\sigma((\nabla\psi_i) \otimes p_i)
+
\psi_i\sigma(\nabla p_i)
].
\end{equation}
Since symmetrization and tensor product commute with pointwise evaluations we have $\sigma((\nabla \psi_i) \otimes p_i)|_{\partial M} = 0$.
Since~$\psi_i = 0$ in~$M \setminus \supp \psi_i$ we have~$\sigma\nabla \psi_i = 0$ in the same open set~$M \setminus \supp \psi_i$.
Together with $p_i = 0$ on $\partial M \cap \supp \psi_i \subseteq \partial M \cap W_i$ vanishing of the covariant derivative $\sigma\nabla \psi_i$ in $M \setminus \supp \psi_i$ implies
\begin{equation}
\begin{split}
\sigma\nabla p|_{\partial M}
&=
\sum_{i = 0}^n
(\psi_i(\sigma\nabla p_i))|_{\partial M}
=
\sum_{i = 0}^n
\psi_i(\sigma\nabla p_i|_{\partial M \cap W_i})
\\
&=
\sum_{i = 0}^n
\psi_i(f|_{\partial M \cap W_i})
=
\sum_{i = 0}^n
(\psi_if)|_{\partial M}
=
f|_{\partial M}.
\end{split}
\end{equation}
Thus~$p$ has the desired properties.
\end{proof}

%%%%
%%%%
%%%%
\subsection{Regularity of the integral function}
\label{subsec:regularity-u}

Let~$(M,g)$ be a simple~$C^{1,1}$ manifold and let~$f\in C^{1,1}(M)$ be a symmetric~$m$-tensor field with~$If = 0$.
Since the main objective is to prove that there is a symmetric~$(m-1)$-tensor field~$p$ on~$M$ so that~$\sigma\nabla p = f$ and by lemma~\ref{lma:boundary-determination} we can find a tensor field~$p \in C^{1,1}(M)$ with this property on the boundary~$\partial M$, we can move to studying tensor fields~$f \in \Lip_0(M)$ vanishing on the boundary.
The following lemma is a special case of~\cite[Lemma 21]{IKPIXRTMLR}.
We record it for the convenience of the reader.

\begin{lemma}
\label{lma:lipschitz-estimate}
Let~$(M,g)$ be a simple~$C^{1,1}$ manifold. Let~$f \in \Lip_0(M)$ be a symmetric~$m$-tensor field on $M$ and let~$u \coloneqq u^f$ be the integral function of~$f$ defined by~\eqref{eq:function-uf}. Then~$u \in \Lip(SM)$.
\end{lemma}

\begin{proof}
Since~$f$ is in $\Lip_0(M)$ the corresponding function on the sphere bundle is in~$\Lip_0(SM)$.
It was shown in~\cite[Lemma 21]{IKPIXRTMLR} that the integral function of a function in~$\Lip_0(SM)$ is again a Lipschitz function on~$SM$.
\end{proof}

Next we prove lemma~\ref{lma:e-ae} which states that if a Lipschitz function~$u$ on~$SM$ arising from of tensor field~$-p$ satisfies the transport equation~$Xu = -f$, then~$\sigma \nabla p = f$ holds pointwise almost everywhere.

\begin{proof}[Proof of lemma~\ref{lma:e-ae}]
Let~$f \in \Lip(M)$ is a symmetric~$m$-tensor field.
Suppose that~$p \in \Lip(M)$ is a symmetric~$m$-tensor field so that the Lipschitz function~$u \coloneqq -\lambda p$ solves the transport equation~$Xu = -f$ everywhere in~$SM$.
We prove that~$\sigma\nabla p = f$ almost everywhere on~$SM$ by proving that
\begin{equation}
\iip{\sigma\nabla p - f}{\eta}_{L^2(M)} = 0
\end{equation}
for all symmetric~$m$-tensor fields~$\eta \in C^1_0(M)$. Since by proposition~\ref{prop:identification} there are positive constants~$c,C >0$ so that
\begin{equation}
c\iip{\lambda h_1}{\lambda h_2}_{L^2(SM)}
\le
\iip{h_1}{h_2}_{L^2(M)}
\le
C\iip{\lambda h_1}{\lambda h_2}_{L^2(SM)}
\end{equation}
for all symmetric~$m$-tensor fields~$h_1,h_2 \in \Lip(M)$ it is enough to prove that
\begin{equation}
\label{eq:proof-of-a-ea1}
\iip{\lambda\sigma\nabla p - \lambda f}{\lambda\eta}_{L^2(SM)} = 0.
\end{equation}

Consider a maximal geodesic~$\gamma$ of~$M$ so that~$\gamma(0) = x \in \partial M$ and~$\dot\gamma(0) = v \in \doo{in}(SM)$.
We denote~$z \coloneqq (x,v)$ and write~$\eta \coloneqq \lambda\eta$ and~$f \coloneqq \lambda f$. 
Furthermore, we denote~$\theta(t) \coloneqq \phi_t(z)$ and~$\eta(t) \coloneqq \eta(\theta(t))$. Then we have
\begin{equation}
\label{eq:a-ea1}
\int_0^{\tau(z)}
(\lambda \sigma\nabla p)(\theta(t))
\eta(t)
\,\d t
=
\int_0^{\tau(z)}
(\nabla p)_{\gamma(t)}(\dot\gamma(t),\dots,\dot\gamma(t))
\eta(t)
\,\d t.
\end{equation}
Since~$\gamma$ is a geodesic, it satisfies~$\nabla_{\dot\gamma}\dot\gamma = 0$. Therefore the Leibniz rule implies
\begin{equation}
\label{eq:a-ea2}
\begin{split}
\int_0^{\tau(z)}
(\nabla p)_{\gamma(t)}(\dot\gamma(t),\dots,\dot\gamma(t))
\eta(t)
\,\d t
&=
\int_0^{\tau(z)}
\partial_t(p_{\gamma(t)}(\dot\gamma(t),\dots,\dot\gamma(t)))
\eta(t)
\,\d t
\\
&=
-\int_0^{\tau(z)}
p_{\gamma(t)}(\dot\gamma(t),\dots,\dot\gamma(t))
\partial_t\eta(t)
\,\d t.
\end{split}
\end{equation}
By assumption
$
u(\theta(t))
=
-p_{\gamma(t)}(\dot\gamma(t),\dots,\dot\gamma(t))
$
for all~$t \in [0,\tau(z)]$ and thus
\begin{equation}
\label{eq:a-ea3}
\begin{split}
-\int_0^{\tau(z)}
p_{\gamma(t)}(\dot\gamma(t),\dots,\dot\gamma(t))
\partial_t\eta(t)
\,\d t
&=
\int_0^{\tau(z)}
u(\theta(t))
\partial_t\eta(t)
\,\d t
\\
&=
-\int_0^{\tau(z)}
\partial_tu(\theta(t))
\eta(t)
\,\d t
\\
&=
\int_0^{\tau(z)}
f(\theta(t))
\eta(t)
\,\d t,
\end{split}
\end{equation}
where the last equality holds since~$Xu = -f$ and~$X$ is the infinitesimal generator of the geodesic flow~$\phi_t$.
Together equations~\eqref{eq:a-ea1},~\eqref{eq:a-ea2} and~\eqref{eq:a-ea3} show that
\begin{equation}
\label{eq:proof-of-e-ae2}
\int_0^{\tau(z)}
(\lambda \sigma\nabla p)(\theta(t))\eta(t)
\,\d t
=
\int_0^{\tau(z)}
f(\theta(t))\eta(t)
\,\d t.
\end{equation}
We integrate~\eqref{eq:proof-of-e-ae2} over~$\doo{in}(SM)$ and use Santaló's formula (lemma~\ref{lma:santalo}) to see that
\begin{equation}
\label{eq:proof-of-e-ae3}
\begin{split}
\int_{SM}
(\lambda \sigma\nabla p)\eta
\,\d\Sigma_g
&=
\int_{\doo{in}(SM)}
\int_0^{\tau(z)}
(\lambda \sigma\nabla p)(\theta(t))\eta(t)
\,\d t
\,\mu\d\Sigma_{j^\ast g}
\\
&=
\int_{\doo{in}(SM)}
\int_0^{\tau(z)}
f(\theta(t))\eta(t)
\,\d t
\,\mu\d\Sigma_{j^\ast g}
\\
&=
\int_{SM}
f\eta
\,\d\Sigma_g.
\end{split}
\end{equation}
Equation~\eqref{eq:proof-of-a-ea1} follows immediately from~\eqref{eq:proof-of-e-ae3}, which finishes the proof.
\end{proof}

%%%%
%%%%
%%%%
\subsection{Regularity of the spherical harmonic components}
\label{subsec:reg-of-sp-harmonics}

In this subsection we use the special form of spherical harmonics and the identification of trace-free tensor fields and spherical harmonics to prove lemma~\ref{lma:regularity-u}.
Also, we prove that the degreewise definition of operators~$X_\pm$ acting on functions on~$SM$ is reasonable by proving that series in~\eqref{eq:xpm} converge absolutely in~$L^2(SM)$.

\begin{proof}[Proof of lemma~\ref{lma:regularity-u}]
Let~$f \in \Lip_0(M)$ be a symmetric~$m$-tensor field with vanishing X-ray transform and let~$u \coloneqq u^f$ be the integral function of~$f$ defined by~\eqref{eq:function-uf}.
The integral function~$u$ is in~$\Lip(SM)$ by lemma~\ref{lma:lipschitz-estimate}.
We prove that the spherical harmonic components~$u_k$ of~$u$ are in~$\OS{0,1}{\infty}{k}$ and that~$u_k|_{\partial(SM)} = 0$.

For a fixed~$x \in M$ the fiber~$S_xM$ is isometric to the Euclidean unit sphere~$S^{n-1} \subseteq \R^n$ via the map
\begin{equation}
s_x \colon S_xM \to S^{n-1},
\quad
s_x(v)
=
g(x)^{1/2}v,
\end{equation}
where $g(x)^{1/2}$ is the unique square root of a positive definite matrix $g(x)$.
Since~$u$ is in $\Lip(SM)$, its restriction~$u_x \coloneqq u(x,\,\cdot\,)$ to~$S_xM$ is in~$\Lip(S_xM)$.
Thus the functions~$\tilde u_x$ on~$S^{n-1}$ corresponding to~$u_x$ via~$s_x$ has a decomposition
\begin{equation}
\tilde u_x
=
\sum_{k=0}^\infty
\iip{\tilde u_x}{\phi_k}_{L^2(S^{n-1})}\phi_k,
\end{equation}
where~$\phi_k$ is the eigenfunction of the Laplacian on~$S^{n-1}$ corresponding to the eigenvalue~$k(k+n-2)$.
Tracing back through~$s_x$ we find a~$L^2(S_xM)$ convergent decomposition
\begin{equation}
u_x
=
\sum_{k=0}^\infty
\iip{u_x}{\psi_k}_{L^2(S_xM)}\psi_k,
\end{equation}
where~$\psi_k(v) = \phi_k(s^{-1}_x(v))$.
On the level of the bundle~$SM$, we denote~$\psi_k(x,v) \coloneqq \phi_k(s^{-1}_x(v))$, and thus get the formula~$u_k = \iip{u}{\psi_k}_{L^2(S_xM)}\psi_k$.
Here~$\psi_k$ is in $C^{1,1}(SM)$, since~$\phi_k$ is in $C^\infty(S^{n-1})$ and the map~$(x,v) \mapsto s_x(v)$ is in~$C^{1,1}(SM)$.
This proves that~$u_k \in \Lip(SM)$.
We note that by lemma~\ref{prop:identification} for all~$k$ there is a symmetric and trace-free~$k$-tensor field~$h_k \in \Lip(M)$ so that~$u_k(x,v) = (h_k)_{j_1\cdots j_k}(x)v^{j_1}\cdots v^{j_k}$. This proves that~$u_k \in \OS{0,1}{\infty}{k}$ for all~$k$, since~$u_k$ is polynomial in~$v$.

Finally, we prove that~$u_k|_{\partial(SM)} = 0$.
Since the X-ray transform of~$f$ is zero, the restriction of~$u$ on the boundary~$\partial(SM)$ is zero.
Thus for any~$x \in \partial M$ we have
\begin{equation}
0 = \norm{u(x,\cdot)}^2_{L^2(S_xM)}
=
\sum_{k = 0}^\infty
\norm{u_k(x,\cdot)}^2_{L^2(S_xM)}.
\end{equation}
Therefore, since~$u_k(x,\cdot) \in C^\infty(S_xM)$, we have~$u_k(x,\cdot) = 0$ pointwise on~$S_xM$ for all~$k$, which implies that~$u_k|_{\partial(SM)} = 0$ for all~$k$.
\end{proof}

\begin{lemma}
\label{lma:xpm-convergence}
Let~$(M,g)$ be a simple~$C^{1,1}$ manifold.
Given~$u \in \HS{1}{2}{SM}$, if~$u = \sum_{k=0}^\infty u_k$ is the spherical harmonic decomposition of~$u$, then the series~$\sum_{k=0}^\infty X_\pm u_k$ converge absolutely in~$L^2(SM)$.
Here we use the convention that~$X_-u_0 = 0$.
\end{lemma}

\begin{proof}
We prove convergence of both of series~$\sum_{k=0}^\infty X_\pm u_k$ at once by proving that
\begin{equation}
\label{eq:xpm-convergence}
\sum_{k=0}^\infty
\norm{X_+u_k}^2_{L^2(SM)}
+
\sum_{k=1}^\infty
\norm{X_-u_k}^2_{L^2(SM)}
\le
\norm{u}^2_{\HS{1}{0}{SM}}.
\end{equation}
The proof of~\eqref{eq:xpm-convergence} is identical to the proofs of~\cite[Lemma 4.4]{PSUIDBTTT} and~\cite[Lemma 5.1]{LRSTTCHM}, where the authors proved that
\begin{equation}
\norm{X_+u}^2_{L^2(SM)}
+
\norm{X_-u}^2_{L^2(SM)}
\le
\norm{Xu}^2_{L^2(SM)}
+
\norm{\grad{h}u}_{L^2(SM)}.
\end{equation}
The major difference to the results in~\cite{PSUIDBTTT} and~\cite{LRSTTCHM} is that we work in non-smooth geometry instead of a smooth geometry, so the tools in the proof have changed.
For completeness, we repeat the arguments in appendix~\ref{app:lainalasku} to document the fact that all steps go through in lower regularity with suitably chosen function spaces.
\end{proof}

\begin{remark}
For~$u \in \HS{1}{2}{SM}$ we defined~$X_\pm u$ to be the series~$\sum_{k=0}^\infty X_\pm u_k$, when~$u = \sum_{k=0}^\infty u_k$ is the spherical harmonic decomposition of~$u$. By lemma~\ref{lma:xpm-convergence} both~$X_+u$ and~$X_-u$ are well defined functions in~$L^2(SM)$ and by orthogonality
\begin{equation}
\norm{X_\pm u}^2_{L^2(SM)}
=
\sum_{k=0}^\infty
\norm{X_\pm u_k}^2_{L^2(SM)}.
\end{equation}
\end{remark}

%%%%
%%%%
%%%%
\section{Energy estimates and a Santaló formula}
\label{sec:estimates-santalo}

In this section we show that the~$L^2$-estimate in lemma~\ref{lma:l2-estimate} follows from the Pestov identity, and we establish the Santaló's formula in low regularity in lemma~\ref{lma:santalo}.

%%%%
%%%%
%%%%
\subsection{Pestov energy identity}

Let~$(M,g)$ be a simple~$C^{1,1}$ manifold. Recall that the global index form~$Q$ of~$(M,g)$ is defined by
\begin{equation}
Q(W) \coloneqq \norm{XW}_{L^2(N)} - \iip{RW}{W}_{L^2(N)}    
\end{equation}
for~$W \in H^1_0(N,X)$.

\begin{lemma}[Pestov identity]
\label{lma:pestov}
Let~$(M,g)$ be a simple~$C^{1,1}$ manifold with almost everywhere non-positive sectional curvature.
If~$u \in \OS{0,1}{\infty}{k}$ and~$u|_{\partial(SM)} = 0$, then
\begin{equation}
\label{eq:pestov}
\norm{\grad{v}Xu}^2_{L^2(N)}
=
Q
\left(
\grad{v}u
\right)
+
(n-1)
\norm{Xu}^2_{L^2(SM)}.
\end{equation}
\end{lemma}

\begin{proof}
Since~$u \in \OS{0,1}{\infty}{k}$, we have~$u \in \Lip_0(SM)$,~$\grad{v}Xu \in L^2(N)$ and~$X\grad{v}u \in L^2(N)$.
It was proved in~\cite[Lemma 9]{IKPIXRTMLR} that the Pestov identity~\eqref{eq:pestov} holds for this class of functions on simple~$C^{1,1}$ manifolds.
\end{proof}

When~$g \in C^\infty$ the estimate in Lemma~\ref{lma:pestov2} was derived in~\cite[Section 6]{IPBRTORO}. We present a proof compatible with low regularity employing the Pestov identity in Lemma~\ref{lma:pestov}.

\begin{lemma}
\label{lma:pestov2}
Let~$(M,g)$ be a simple~$C^{1,1}$ manifold with almost everywhere non-positive sectional curvature.
If~$u \in \OS{0,1}{\infty}{k}$ and~$u|_{\partial(SM)} = 0$, then
\begin{equation}
\iip{Xu}{[X,\lap{v}]u}_{L^2(SM)}
\le
0.
\end{equation}
\end{lemma}

\begin{proof}
Since the sectional curvature of~$(M,g)$ is almost everywhere non-positive,
$Q(W) \ge \norm{XW}^2$ for all~$W \in H^1_0(N,X)$ and we have
\begin{equation}
\label{eq:pestov1}
\norm{\grad{v}Xu}^2_{L^2(N)}
\ge
\norm{X\grad{v}u}^2_{L^2(N)}
+
(n-1)
\norm{Xu}^2_{L^2(SM)}
\end{equation}
by the Pestov identity (lemma~\ref{lma:pestov}). On the other hand, using commutator formulas from proposition~\ref{prop:commutator-formulas} we see that
\begin{equation}
\label{eq:pestov2}
\begin{split}
\norm{X\grad{v}u}^2
&=
\norm{\grad{v}Xu - \grad{h}u}^2
\\
&=
\norm{\grad{v}Xu}^2
-
2\iip{\grad{v}Xu}{\grad{h}u}
+
\norm{\grad{h}u}^2
\\
&=
\norm{\grad{v}Xu}^2
+
\iip{Xu}{2\dive{v}\grad{h}u}
+
\norm{\grad{h}u}^2.
\end{split}
\end{equation}
Combining estimate~\eqref{eq:pestov1} and equation~\eqref{eq:pestov2} and applying the commutator formula~\eqref{eq:comm-laplace-x} we get
\begin{equation}
\begin{split}
0 &\ge
\iip{Xu}{2\dive{v}\grad{h}u}
+
\norm{\grad{h}u}^2
+
(n-1)
\norm{Xu}^2
\\
&\ge
\iip{Xu}{2\dive{v}\grad{h}u + (n-1)Xu}
\\
&=
\iip{Xu}{[X,\lap{v}]u}
\end{split}
\end{equation}
as claimed.
\end{proof}

\begin{lemma}
\label{lma:l2-xplus-to-xminus}
Let~$(M,g)$ be a simple~$C^{1,1}$ manifold with almost everywhere non-positive sectional curvature.
Suppose that~$f \in \Lip_0(M)$ is a symmetric~$m$-tensor field on~$M$ with vanishing X-ray transform~$If$.
Let~$u \coloneqq u^f$ be the integral function of~$f$ defined by~\eqref{eq:function-uf}.
If~$k \ge m$ or~$k \equiv m \pmod 2$, we have
\begin{equation}
\label{eq:l2-xplus-to-xminus}
\norm{X_+u_k}^2_{L^2(SM)}
=
\norm{X_-u_{k+2}}^2_{L^2(SM)}.
\end{equation}
\end{lemma}

\begin{proof}
Since~$f \in \Lip_0(M)$ and the X-ray transform of~$f$ vanishes, we have~$u \in \Lip_0(SM)$ by lemma~\ref{lma:lipschitz-estimate}.
By the fundamental theorem of calculus~$u$ solves~$Xu = -f$ and projecting this transport equation onto spherical harmonic degree~$k+1$ gives
\begin{equation}
-f_{k+1}
=
X_+u_k
+
X_-u_{k+2}
\end{equation}
If~$k \ge m$ or~$k \equiv m \pmod 2$, then~$f_{k+1} = 0$ and the claim~\eqref{eq:l2-xplus-to-xminus} follows by taking~$L^2$-norms.
\end{proof}

Recall that the constants~$C(n,k)$ and~$B(n,l,k)$ in lemma~\ref{lma:l2-estimate} are
\begin{equation}
C(n,k)
\coloneqq
\frac{2k+n-1}{2k+n-3}
\quad\text{and}\quad
B(n,l,k)
\coloneqq
\prod_{p = 1}^lC(n,k+2p).
\end{equation}

\begin{lemma}
\label{lma:estimate-xminus-to-xplus}
Let~$(M,g)$ be a simple~$C^{1,1}$ manifold with almost everywhere non-positive sectional curvature.
Suppose that~$f \in \Lip_0(M)$ is a symmetric~$m$-tensor field with~$If = 0$.
Let~$u \coloneqq u^f$ be integral function of~$f$ defined by~\eqref{eq:function-uf}.
If~$2k+n-3>0$, we have
\begin{equation}
\label{eq:estimate-xminus-to-xplus}
\norm{X_-u_k}^2_{L^2(SM)}
\le
C(n,k)
\norm{X_+u_k}^2_{L^2(SM)},
\end{equation}
where~$u_k$ are the spherical harmonic components of~$u$.
\end{lemma}

\begin{proof}
Let~$2k+n-3>0$.
Since~$u_k \in \OS{0,1}{\infty}{k}$ by lemma~\ref{lma:regularity-u}, we can use lemma~\ref{lma:pestov2}, which together with commutator formulas in~\ref{lma:commutators2} gives
\begin{equation}
\begin{split}
(2k+n-1)\norm{X_+u_k}^2
&\ge
(2k+n-1)\norm{X_+u_k}^2
+
\iip{Xu_k}{[X,\lap{v}]u_k}
\\
&=
(2k+n-1)\norm{X_+u_k}^2
+
\iip{X_+u_k}{[X_+,\lap{v}]u_k}
\\
&\quad+
\iip{X_-u_k}{[X_-,\lap{v}]u_k}
\\
&=
(2k+n-3)\norm{X_-u_k}^2.
\end{split}
\end{equation}
Dividing by~$2k+n-3>0$ proves the claimed estimate~\eqref{eq:estimate-xminus-to-xplus}.
\end{proof}

\begin{proof}[Proof of lemma~\ref{lma:l2-estimate}]
Let~$f \in \Lip_0(M)$ be a symmetric~$m$-tensor field so that~$If = 0$ and denote by~$u \coloneqq u^f$ its integral function defined by~\eqref{eq:function-uf}. Let~$k \ge m$. By lemma~\ref{lma:regularity-u} we have~$u \in \OS{0,1}{\infty}{k}$ and thus lemmas~\ref{lma:l2-xplus-to-xminus} and~\ref{lma:estimate-xminus-to-xplus} we get
\begin{equation}
\norm{X_+u_k}^2_{L^2(SM)}
=
\norm{X_-u_{k+2}}^2_{L^2(SM)}
\le
C(n,k+2)
\norm{X_+u_{k+2}}^2_{L^2(SM)}.
\end{equation}
Iterating lemmas~\ref{lma:l2-xplus-to-xminus} and~\ref{lma:estimate-xminus-to-xplus} a total of~$l \in \N$ times yields
\begin{equation}
\norm{X_+u_k}^2
\le
\norm{X_+u_{k+2l}}^2
\prod_{p = 1}^lC(n,k+2p)
=
B(n,l,k)
\norm{X_+u_{k+2l}}^2
\end{equation}
as claimed.
\end{proof}

%%%%
%%%%
%%%%
\subsection{Santaló's formula}

The proof of Santaló's formula on a smooth simple manifolds~$(M,g)$ is based on the so called Liouville's theorem and can be found e.g. in~\cite{PSUGIPETD}.
We give a similar proof of the formula on a simple~$C^{1,1}$ manifold based on the following formulation of Liouville's theorem.

\begin{lemma}
\label{lma:liouville}
Let~$(M,g)$ be a simple~$C^{1,1}$ manifold.
Denote by~$L_X$ the Lie derivative into the direction of the geodesic vector field~$X$ on~$SM$. Then for any~$u \in \Lip(SM)$ it holds that
\begin{equation}
\int_{SM}
uL_X(\d\Sigma_g)
=
0.
\end{equation}
\end{lemma}

The proof of lemma~\ref{lma:liouville} is based on smooth approximation of the Riemannian metric~$g$ and can be found in Appendix~\ref{sec:liouville}.

If~$\nu$ is the inner unit normal vector field to~$\partial M$, let~$\mu(x,v) \coloneqq \ip{\nu(x)}{v}_{g(x)}$ for all~$(x,v) \in SM$.
If~$\omega$ is a differential~$k$-form on~$SM$, then denote by~$i_X\omega$ the contraction of~$\omega$ with the geodesic vector field~$X$.
That is, for any vector fields~$Y_1,\dots,Y_{k-1}$ on~$SM$, we define $i_X\omega$ by letting
$i_X\omega(Y_1,\dots,Y_{k-1}) = \omega(X,Y_1,\dots,Y_{k-1})$.

\begin{lemma}[Santaló's formula]
\label{lma:santalo}
Let~$(M,g)$ be a simple~$C^{1,1}$ manifold.
For any function~$f \in \Lip_0(SM)$ the integral of~$f$ over~$SM$ with respect to~$\d\Sigma_g$ can be written as
\begin{equation}
\label{eq:santalo}
\int_{SM}
f
\,d\Sigma_g
=
\int_{\doo{in}SM}
\int_0^{\tau(z)}
f(\phi_t(z))
\,\d t
\,\mu(z)
\d\Sigma_{j^{\ast}g}.
\end{equation}
Here~$j \colon \partial(SM) \to SM$ is the inclusion map and~$j^\ast g$ is the Riemannian metric of~$\partial M$ induced by the inclusion~$j$.
\end{lemma}

\begin{proof}
Let~$f \in \Lip_0(SM)$ and consider its integral function~$u \coloneqq u^f$.
The integral function satisfies~$Xu = -f$ and~$u \in \Lip(SM)$ by lemma~\ref{lma:lipschitz-estimate}.
By Cartan's formula we have
\begin{equation}
\label{eq:cartan}
\int_{SM}
L_X(u\,\d\Sigma)
=
\int_{SM}
i_Xd(u\,\d\Sigma)
+
\int_{SM}
d(i_Xu\,\d\Sigma),
\end{equation}
where~$d$ is the exterior derivative.
Since~$u\,\d\Sigma$ is a volume form, the first term on the right in~\eqref{eq:cartan} vanishes.
By Stoke's theorem
\begin{equation}
\int_{SM}
d(i_Xu\,\d\Sigma_g)
=
\int_{\partial(SM)}
j^{\ast}
(ui_X\d\Sigma_g).
\end{equation}
As in the smooth case (\cite[Proposition 3.6.6.]{PSUGIPETD}), we compute that
\begin{equation}
\begin{split}
\int_{\partial(SM)}
j^{\ast}
(ui_X\d\Sigma_g)
&=
\int_{SM}
(j^{\ast}u)(j^{\ast}i_X\d\Sigma_g)
\\
&=
\int_{SM}
(j^\ast u)
\ip{X}{\nu}
\d\Sigma_{j^{\ast}g}
\\
&=
\int_{SM}
(j^\ast u)
\mu
\,\d\Sigma_{j^{\ast}g}.
\end{split}
\end{equation}
Finally, since~$j^\ast u$ is merely a restriction to the boundary, we invoke the definition of~$u$ and lemma~\ref{lma:liouville} to see that
\begin{equation}
\label{eq:proof-of-santalo}
\begin{split}
\int_{SM}
f
\,\d\Sigma_g
&=
\int_{SM}
L_X(u)\,\d\Sigma
\\
&=
\int_{SM}
L_X(u\,\d\Sigma)
-
\int_{SM}
uL_X(\d\Sigma)
\\
&=
\int_{SM}
L_X(u\,\d\Sigma)
\\
&=
\int_{\partial(SM)}
(j^\ast u)
\mu\,\d\Sigma_{j^\ast g}
\\
&=
\int_{\partial(SM)}
\int_0^{\tau(z)}
f(\phi_t(z))
\,\d t
\mu\,\d\Sigma_{j^\ast g}.
\end{split}
\end{equation}
Since~$\tau(z) = 0$ for~$z \notin \doo{in}(SM)$ the claim~\eqref{eq:santalo} follows at once from~\eqref{eq:proof-of-santalo}.
\end{proof}

%%%%
%%%%
%%%%
\section{Friedrich's inequalities}
\label{sec:friedrichs}

In this section we prove that~$L^2$-norms of scalar functions on~$SM$ and sections of the bundle~$N$ are bounded above by constant multiples of~$L^2$-norms of their derivatives along the geodesic flow.
We call these estimates Friedrich's inequalities on~$SM$.
We apply the inequalities to prove lemma~\ref{lma:injectivity-xplus}.

\begin{lemma}
\label{lma:friedrichs}
Let~$(M,g)$ be a simple~$C^{1,1}$ manifold with almost everywhere non-positive sectional curvature. Let~$d$ be the diameter of~$M$. Then
\begin{equation}
\label{eq:friedrichs}
d^2\norm{Xu}^2_{L^2(SM)}
\ge
\norm{u}^2_{L^2(SM)}
\quad\text{and}\quad
d^2\norm{XW}^2_{L^2(N)}
\ge
\norm{W}^2_{L^2(N)}
\end{equation}
for any~$u \in H^1_0(SM)$ and~$W \in H^1_0(N,X)$.
\end{lemma}

\begin{proof}
First, we prove the inequality for functions. By density is enough to consider the case~$u \in C^1_0(SM)$. By Santaló's formula (lemma~\ref{lma:santalo}) we can write
\begin{equation}
\norm{Xu}^2_{L^2(SM)}
=
\int_{\doo{in}(SM)}
\int_0^{\tau(z)}
\abs{Xu(\phi_t(z)}^2
\,\d t
\,\mu\d\Sigma_{j^\ast g},
\end{equation}
where~$j \colon \partial(SM) \to SM$ is the inclusion.
Let us denote~$u_z(t) \coloneqq u(\phi_t(z))$. Then~$u_z \in H^1_0([0,\tau(z)])$ and we have
\begin{equation}
\label{eq:xu-santalo}
Xu(\phi_t(z))
=
\frac{\d}{\d s}
u(\phi_{t+s}(z))
\bigg|_{s=0}
=
\frac{\d}{\d s}
u_z(t+s)
\bigg|_{s=0}
=
\dot{u}_z(t).
\end{equation}
By the usual Friedrich's inequality of~$H^1_0([0,\tau(z)])$ we see that
\begin{equation}
\label{eq:uz-friedrichs}
d^2
\int_0^{\tau(z)}
\abs{\dot{u}_z(t)}^2
\,\d t
\ge
\tau(z)^2
\int_0^{\tau(z)}
\abs{\dot{u}_z(t)}^2
\,\d t
\ge
\int_0^{\tau(z)}
\abs{u_z(t)}^2
\,\d t.
\end{equation}
Combining equation~\eqref{eq:xu-santalo} with inequality~\eqref{eq:uz-friedrichs} we get
\begin{equation}
\begin{split}
d^2\norm{Xu}^2_{L^2(SM)}
&\ge
d^2
\int_{\doo{in}(SM)}
\int_0^{\tau(z)}
\abs{\dot{u}_z(t)}^2
\,\d t
\,\mu\d\Sigma_{j^\ast g}
\\
&\ge
\int_{\doo{in}(SM)}
\int_0^{\tau(z)}
\abs{u_z(t)}^2
\,\d t
\,\mu\d\Sigma_{j^\ast g}
\\
&=
\norm{u}^2_{L^2(SM)},
\end{split}
\end{equation}
which is the claimed inequality for functions.

Next, we prove the inequality for sections of the bundle~$N$.
Let~$W \in H^1_0(N,X)$.
In this case Santaló's formulas (lemma~\ref{lma:santalo}) gives
\begin{equation}
\label{eq:xw-santalo}
\norm{XW}^2_{L^2(SM)}
=
\int_{\doo{in}(SM)}
\int_0^{\tau(z)}
\abs{XW(\phi_t(z))}^2_{g}
\,\d t
\,\mu(z)
\,\d\Sigma_{\partial(SM)}.
\end{equation}
We let~$W_z(t) \coloneqq W(\phi_t(z))$. Then~$W_z(t)$ is a~$H^1_0$ vector field along~$\gamma_z$ and it holds that~$XW(\phi_t(z)) = D_tW_z(t)$.
Choose a parallel frame $(E_1,\dots,E_n)$ along~$\gamma_z$.
Then we have~$D_tW_z = \dot W^i_zE_i$, when~$W_z = W^i_zE_i$.
Since~$W_z$ is a~$H^1_0$ vector field along~$\gamma_z$ we have~$W^i_z \in H^1_0([0,\tau(z)])$ for all~$i$.
Thus we read from equation~\eqref{eq:uz-friedrichs} that
\begin{equation}
\label{eq:wz-friedrich}
d^2
\int_0^{\tau(z)}
\abs{\dot W^i_z}^2
\,\d t
\ge
\int_0^{\tau(z)}
\abs{W^i_z}^2
\,\d t.
\end{equation}
From equations~\eqref{eq:xw-santalo} and~\eqref{eq:wz-friedrich} we see that
\begin{equation}
\begin{split}
d^2
\norm{XW}^2_{L^2(N)}
&=
d^2
\int_{\doo{in}(SM)}
\int_0^{\tau(z)}
\abs{D_tW_z(t)}^2_{g}
\,\d t
\,\mu(z)
\,\d\Sigma_{\partial(SM)}
\\
&=
d^2
\sum_{i = 1}^n
\int_{\doo{in}(SM)}
\int_0^{\tau(z)}
\abs{\dot W_z^i(t)}^2
\,\d t
\,\mu(z)
\,\d\Sigma_{\partial(SM)}
\\
&\ge
\sum_{i = 1}^n
\int_{\doo{in}(SM)}
\int_0^{\tau(z)}
\abs{W_z^i(t)}^2
\,\d t
\,\mu(z)
\,\d\Sigma_{\partial(SM)}
\\
&=
\norm{W}^2_{L^2(N)},
\end{split}
\end{equation}
which is the second claimed inequality.
\end{proof}

\begin{proof}[Proof of lemma~\ref{lma:injectivity-xplus}]
Let~$u \in \OS{0,1}{\infty}{k}$ be so that~$u|_{\partial(SM)} = 0$ and~$X_+u = 0$.
By lemma~\ref{lma:commutators2} we have
\begin{equation}
\label{eq:xp-injective}
\begin{split}
(2k+n-3)\norm{X_-u}^2
&=
-(2k+n-1)\norm{X_+u}^2
+
(2k+n-3)\norm{X_-u}^2
\\
&=
\iip{[X_+,\lap{v}]u}{X_+u}
+
\iip{[X_-,\lap{v}]u}{X_-u}
\\
&=
\iip{[X_+,\lap{v}]u}{Xu}
+
\iip{[X_-,\lap{v}]u}{Xu}
\\
&=
\iip{[X,\lap{v}]u}{Xu}.
\end{split}
\end{equation}
The last inner product in~\eqref{eq:xp-injective} is non-positive by lemma~\ref{lma:pestov2}.
Thus~$X_-u = 0$ almost everywhere on~$SM$.
Let~$d$ be the diameter of~$M$.
Lemma~\ref{lma:friedrichs} then provides
\begin{equation}
\norm{u}^2_{L^2(SM)}
\le
d^2\norm{Xu}^2_{L^2(SM)}
=
d^2
(\norm{X_+u}^2_{L^2(SM)}
+
\norm{X_-u}^2_{L^2(SM)})
=
0.
\end{equation}
Thus~$u = 0$ almost everywhere on~$SM$, but since~$u$ is continuous we have shown that~$u = 0$ everywhere on~$SM$.
\end{proof}

%%%%
%%%%
%%%%

Even though we do not need the result, we next show for completeness that there are no conjugate points in the sense of the global index form~$Q$ when the sectional curvature is non-positive.

\begin{proposition}
Let~$M$ be the closed Euclidean unit ball in~$\R^n$.
Suppose that~$M$ comes equipped with a~$C^{1,1}$ Riemannian metric~$g$ so that the sectional curvature of~$(M,g)$ is almost everywhere non-positive.
Then there is~$\varepsilon > 0$ so that~$Q(W) \ge \varepsilon\norm{W}^2_{L^2(N)}$ for all~$W \in H^1_0(N,X)$.
\end{proposition}

\begin{proof}
Since the sectional curvature is almost everywhere non-positive,
\begin{equation}
\iip{RW}{W}_{L^2(N)}
=
\int_{(x,v) \in SM}
\ip{R(W(x,v),v)v}{W(x,v)}_g
\,\d\Sigma_g
\le
0
\end{equation}
for all~$W \in H^1_0(N,X)$, since~$W(x,v)$ and~$v$ are always orthogonal.
Thus $Q(W) \ge \norm{XW}^2_{L^2(N)}$ for all~$W \in H^1_0(N,X)$.
Then it follows from lemma~\ref{lma:friedrichs} that for all~$W \in H^1_0(N,X)$ we have
\begin{equation}
Q(W)
\ge
\norm{XW}^2_{L^2(N)}
\ge
\frac{1}{d^2}
\norm{W}^2_{L^2(N)}.
\end{equation}
We take $\eps = 1/d^2$ which finishes the proof.
\end{proof}

%%%%
%%%%
%%%%
\appendix

%%%%
%%%%
%%%%

\section{Completion of the proof of boundary determination}
\label{app:construction-potential}

We complete the details in the proof of lemma~\ref{lma:local-bnd-determination} by proving items~\ref{item:constr-p-1} and~\ref{item:constr-p-2}. Recall that we work in local coordinates~$\phi \colon W \to \R^n$ so that
\begin{equation}
\phi(W \cap \partial M)
=
\{x^n = 0\},
\quad\text{and}\quad
\phi(W \cap M^{\mathrm{int}})
=
\{x^n > 0\}.
\end{equation}
We denote~$\hat x = (x^1,\dots,x^{n-1})$. The local tensor field~$p$ is defined in these coordinates by
\begin{equation}
p_{j_1\cdots j_ln \cdots n}(\hat x, x^n)
=
\frac{m}{m-l}
x^n
f_{j_1\cdots j_ln \cdots n}(\hat x, 0),
\end{equation}
where~$n$ appears~$m-1-l$ times in~$p_{j_1\cdots j_ln \cdots n}$ and~$m-l$ times in~$f_{j_1\cdots j_ln \cdots n}$.

First we prove item~\ref{item:constr-p-1}. We begin by proving that $f_x(v,\dots,v) = 0$ for all $v \in S_x(W \cap \partial M)$ and $x \in W \cap \partial M$. Given $v \in S_x(W \cap \partial M)$ we choose a sequence $(v_k)$ of vectors $v_k \in S_x(W \cap \partial M)$ so that $\tau(x,v_k)>0$, and $\tau(x,v_k) \to 0$ and $v_k \to v$ when $k \to \infty$. Such a sequence of vectors exists by $C^{1,1}$ simplicity as proved in~\cite[Lemma 23]{IKPIXRTMLR}. Since the lengths of the geodesics corresponding to $(x,v_k)$ become arbitrarily short and $If = 0$, we find that
\begin{equation}
\begin{split}
f_x(v,\dots,v)
&=
\lim_{k \to \infty}
\frac{1}{\tau(x,v_k)}
\int_0^{\tau(x,v_k)}
f(\phi_t(x,v_k))
\,\d t
\\
&=
\lim_{k \to \infty}
\frac{If(x,v_k)}{\tau(x,v_k)}
\\
&= 0.
\end{split}
\end{equation}
We have shown that $f_x(v,\dots,v) = 0$ for all $v \in S_x(W \cap \partial M)$. Next, we prove that $f_{j_1 \cdots j_m}(\hat x,0) = 0$ in $W \cap \partial M$ for all $j_1,\dots,j_m \in \{1,\dots,n-1\}$.

Let $\iota \colon \partial M \to M$ be the inclusion map. The pullback $\iota^\ast f$ is an $m$-tensor field on $\partial M$. Since $f_x(v,\dots,v) = 0$ for all $v \in S_x(W \cap \partial M)$ we have $(\iota^\ast f)_x(v,\dots,v) = 0$ for all $v \in S_x(W \cap \partial M)$. Then a fiberwise computation~\cite[Lemma 2.4]{DSCKSTFRM} shows that
\begin{equation}
0
=
\int_{W \cap \partial M}
(\iota^\ast f)_x(v,\dots,v)^2
\,\d S_x
=
C_{m,n-1}\abs{\iota^\ast f}^2_{g(x)}
\end{equation}
for all $x \in W \cap \partial M$. We have shown that $\iota^\ast f|_{W \cap \partial M} = 0$ which written in the coordinates in $W$ gives $f_{j_1\cdots j_m}(\hat x,0) = 0$ for all $j_1,\dots,j_m \in \{1,\dots,n-1\}$. We have proved item~\ref{item:constr-p-1}.

We proceed to proving item~\ref{item:constr-p-2}.
Let~$l \in \{0,\dots,m-1\}$ and~$j_1,\dots,j_l \in \{1,\dots,n-1\}$.
To compute the restriction to boundary of the component functions of~$\sigma\nabla p$, we first compute~$\nabla_np_{j_1\cdots j_ln \cdots n}(\hat x,0)$ and~$\nabla_{j_s}p_{j_1\cdots \widehat{j_s}\cdots j_ln \cdots n}(\hat x,0)$.
We have
\begin{equation}
\label{eq:bnd-det1}
\begin{split}
\nabla_np_{j_1\cdots j_ln \cdots n}
&=
\partial_np_{j_1\cdots j_ln \cdots n}
\\
&\quad
-
\sum_{s=1}^l\Gamma^k_{nj_s}
p_{j_1\cdots k\cdots j_ln \cdots n}
\\
&\quad
-
\sum_{s=l+1}^{m-1}\Gamma^k_{nn}
p_{j_1\cdots j_ln\cdots k \cdots n}.
\end{split}
\end{equation}
Thus by the construction of~$p$ we find that
\begin{equation}
\label{eq:bdn-det2}
\begin{split}
\nabla_np_{j_1\cdots j_ln \cdots n}(\hat x,x^n)
&=
\frac{m}{m-l}
f_{j_1\cdots j_ln \cdots nn}(\hat x,0)
\\
&\quad-\frac{m}{m-l}x^n
\sum_{s=1}^l\Gamma^k_{nj_s}
f_{j_1\cdots k\cdots j_ln \cdots nn}(\hat x,0)
\\
&\quad-\frac{m}{m-l}x^n
\sum_{s=l+1}^{m-1}\Gamma^k_{nn}
f_{j_1\cdots j_ln\cdots k \cdots nn}(\hat x,0).
\end{split}
\end{equation}
On the boundary~$\{x^n = 0\}$ equation~\eqref{eq:bdn-det2} reduces to
\begin{equation}
\label{eq:bnd-det5}
\nabla_np_{j_1\cdots j_ln \cdots n}(\hat x,0)
=
\frac{m}{m-l}
f_{j_1\cdots j_ln \cdots nn}(\hat x,0).
\end{equation}
As in equation~\eqref{eq:bnd-det1} we have
\begin{equation}
\label{eq:bnd-det3}
\begin{split}
\nabla_{j_s}p_{j_1\cdots \widehat{j_s} \cdots j_ln \cdots n}
&=
\partial_{j_s}p_{j_1\cdots \widehat{j_s} \cdots j_ln \cdots n}
\\
&\quad-
\sum_{r=1}^{l-1}\Gamma^k_{j_sj_r}
p_{j_1\cdots k\cdots j_ln \cdots n}
\\
&\quad-
\sum_{r=l}^{m-1}\Gamma^k_{j_sj_r}
p_{j_1\cdots j_ln\cdots k \cdots n}.
\end{split}
\end{equation}
By the construction of~$p$, equation~\eqref{eq:bnd-det3} gives
\begin{equation}
\label{eq:bnd-det7}
\begin{split}
\nabla_{j_s}p_{j_1\cdots \widehat{j_s} \cdots j_ln \cdots n}(\hat x,x^n)
&=
\frac{m}{m-l}x^n
\partial_{j_s}f_{j_1\cdots \widehat{j_s} \cdots j_ln \cdots nn}(\hat x,0)
\\
&\quad
-
\frac{m}{m-l}x^n
\sum_{r=1}^{l-1}\Gamma^k_{j_sj_r}
f_{j_1\cdots k\cdots j_ln \cdots nn}(\hat x,0)
\\
&\quad
-
\frac{m}{m-l}x^n
\sum_{r=l}^{m-1}\Gamma^k_{j_sn}
f_{j_1\cdots j_ln\cdots k \cdots nn}(\hat x,0).
\end{split}
\end{equation}
Therefore on the boundary~$\{x^n = 0\}$ we get
\begin{equation}
\label{eq:bnd-det6}
\nabla_{j_s}p_{j_1\cdots \widehat{j_s} \cdots j_ln \cdots n}(\hat x,0) = 0.
\end{equation}
Now we are ready to compute~$(\sigma\nabla p)_{j_1\dots j_ln \cdots n}$, when~$l \in \{0,\dots,m-1\}$.
Denote~$j_{l+1} = \dots = j_m = n$.
There are~$(m-l)(m-1)!$ permutations~$\pi$ of~$\{1,\dots,m\}$ so that~$j_{\pi(1)} = n$, when no restrictions are set on the remaining indices~$j_{\pi(2)},\dots,j_{\pi(m)}$.
Thus using symmetry of~$p$ we find that
\begin{equation}
\label{eq:bnd-det4}
\begin{split}
(\sigma \nabla p)_{j_1\cdots j_ln \cdots n}
&=
\frac{(m-l)(m-1)!}{m!}
\nabla_np_{j_1\cdots j_ln \cdots n}
+
\frac{(m-1)!}{m!}
\sum_{s=1}^l\nabla_{j_s}p_{j_1\cdots \widehat{j_s} \cdots j_ln \cdots n}
\\
&=
\frac{m-l}{m}\nabla_np_{j_1\cdots j_ln \cdots n}
+
\frac{1}{m}
\sum_{s=1}^l
\nabla_{j_s}p_{j_1\cdots \widehat{j_s} \cdots j_ln \cdots n}.
\end{split}
\end{equation}
Evaluating~\eqref{eq:bnd-det4} on the boundary~$\{x^n = 0\}$ and substituting~\eqref{eq:bnd-det5} and~\eqref{eq:bnd-det6} results in
\begin{equation}
\begin{split}
(\sigma \nabla p)_{j_1\dots j_ln \cdots n}(\hat x,0)
&=
f_{j_1\cdots j_ln \cdots n}(\hat x,0).
\end{split}
\end{equation}
The last step is to prove that
\begin{equation}
(\sigma\nabla p)_{j_1\cdots j_m}(\hat x,0)
=
f_{j_1\cdots j_m}(\hat x,0)
\end{equation}
when~$j_1,\dots,j_m \in \{1,\dots,n-1\}$.
By the definition of the symmetrized covariant derivative
\begin{equation}
(\sigma \nabla p)_{j_1\cdots j_m}
=
\frac{1}{m!}\sum_{\pi}
\nabla_{j_{\pi(1)}}p_{j_{\pi(2)}\cdots j_{\pi(m)}}
\end{equation}
where the summation is over all permutations $\pi$ of~$\{1,\dots,m\}$.
Since~$j_{\pi(k)} < n$ for all~$k \in \{1,\dots,m\}$, we can compute as in~\eqref{eq:bnd-det7} to see that
\begin{equation}
\nabla_{j_{\pi(1)}}p_{j_{\pi(2)}\cdots j_{\pi(m)}}|_{x^n = 0} = 0
\end{equation}
for all permutations~$\pi$ of~$\{1,\dots,m\}$. Thus
\begin{equation}
(\sigma \nabla p)_{j_1\cdots j_m}|_{x^n = 0}
=
0
=
f_{j_1\cdots j_m}|_{x^n = 0}.
\end{equation}
We have finally used item~\ref{item:constr-p-1} of the proof,
where we proved that $f_{j_1 \cdots j_m}(\hat x,0) = 0$ for all $j_1,\dots,j_m \in \{1,\dots,n-1\}$.
This concludes the proof item~\ref{item:constr-p-2} and thus the proof of lemma~\ref{lma:local-bnd-determination} is completed.

\section{A regularity computation}
\label{app:lainalasku}

The following calculation completes the proof of lemma~\ref{lma:regularity-u}.
It is based on the proofs of~\cite[Lemma 4.4]{PSUIDBTTT} and~\cite[Lemma 5.1]{LRSTTCHM}.

Let~$u \in \HS{1}{2}{SM}$ and let~$w_k \in \OS{1}{\infty}{k}$ be so that~$w_k|_{\partial(SM)} = 0$. Then~$\grad{h}u \in \HS{1}{1}{SM}$ and thus
\begin{equation}
\label{eq:proof-of-convergence1}
\iip{\grad{h}u}{\grad{v}w_k}_{L^2(N)}
=
-\iip{\dive{v}\grad{h}u}{w_k}_{L^2(N)}.
\end{equation}
Using propostion~\ref{prop:commutator-formulas} the right side can rewritten as
\begin{equation}
\label{eq:proof-of-convergence2}
-\iip{\dive{v}\grad{h}u}{w_k}
=
-\frac12
\iip{[X,\lap{v}]u}{w_k}
+
\frac{n-1}{2}
\iip{Xu}{w_k}.
\end{equation}
If~$u_k \in \LS{1}{2}{k}$ are the spherical harmonic components of~$u$, then by orthogonality and lemma~\ref{lma:commutators2} we have
\begin{equation}
\label{eq:proof-of-convergence3}
\begin{split}
\iip{[X,\lap{v}]u}{w_k}
&=
\iip{[X_+,\lap{v}]u_{k-1} + [X_-,\lap{v}]u_{k+1}}{w_k}
\\
&=
\iip{
-\frac{2k+n-3}{2}X_+u_{k-1}
+
\frac{2k+n-1}{2}X_-u_{k+1}
}
{
w_k
}.
\end{split}
\end{equation}
Together equations~\eqref{eq:proof-of-convergence1},~\eqref{eq:proof-of-convergence2} and~\eqref{eq:proof-of-convergence3} show that
\begin{equation}
\label{eq:proof-of-convergence4}
\iip{\grad{h}u}{\grad{v}w_k}
=
\iip{(k+n-2)X_+u_{k-1}-kX_-u_{k+1}}{w_k}.
\end{equation}

Then we let~$w \in \C{1}{2}{SM}$ so that~$w|_{\partial(SM)} = 0$.
If we decompose~$w$ into spherical harmonics~$w_k$, then~$w_k \in \OS{1}{\infty}{k}$.
We sum equation~\eqref{eq:proof-of-convergence4} over~$k \in \N$ and use~$k(k+n-2)w_k = \lap{v} w_k$ to get
\begin{equation}
\begin{split}
\iip{\grad{h}u}{\grad{v}w}
&=
\sum_{k=0}^\infty
\iip{(k+n-2)X_+u_{k-1}+kX_-u_{k+1}}{w_k}
\\
&=
\sum_{k=0}^\infty
\iip{\frac{1}{k}X_+u_{k-1}+\frac{1}{k+n-2}X_-u_{k+1}}{\lap{v} w_k}
\\
&=
\iip{
\sum_{k=0}^\infty
\grad{v}
\left[\frac{1}{k}X_+u_{k-1}+\frac{1}{k+n-2}X_-u_{k+1}
\right]
}
{\grad{v} w_k}.
\end{split}
\end{equation}
Thus there is~$W(u) \in \HS{0}{1}{N}$ so that~$\dive{v}(W(u)) = 0$ and
\begin{equation}
\label{eq:proof-of-convergence5}
\grad{h}u
=
\sum_{k=0}^\infty
\grad{v}
\left[
\frac{1}{k}X_+u_{k-1}+\frac{1}{k+n-2}X_-u_{k+1}
\right]
+
W(u).
\end{equation}
It follows from the eigenvalue property that
\begin{equation}
\norm{\grad{v}u_k}^2_{L^2(N)}
=
k(k+n-2)\norm{u_k}^2_{L^2(SM)}.
\end{equation}
Thus equation~\eqref{eq:proof-of-convergence5} yields
\begin{equation}
\label{eq:proof-of-convergence6}
\begin{split}
\norm{\grad{h}u}^2
&=
\sum_{k=0}^\infty
k(k+n-2)
\norm{
\frac{1}{k}X_+u_{k-1}+\frac{1}{k+n-2}X_-u_{k+1}
}^2
+
\norm{W(u)}^2
\\
&=
\sum_{k=0}^\infty
\left(
\frac{k+n-2}{k}
\norm{X_+u_{k-1}}^2
-
2\iip{X_+u_{k-1}}{X_-u_{k+1}}
\right.
\\
&\quad\quad\quad+
\left.
\frac{k}{k+n-2}
\norm{X_-u_{k+1}}^2
\right)
+
\norm{W(u)}^2.
\end{split}
\end{equation}
Again, by orthogonality we have
\begin{equation}
\label{eq:proof-of-convergence7}
\begin{split}
\norm{Xu}^2
&=
\sum_{k=0}^\infty\norm{X_+u_{k-1} + X_-u_{k+1}}^2
\\
&=
\sum_{k=0}^\infty
\left(
\norm{X_+u_{k-1}}^2
+
2\iip{X_+u_{k-1}}{X_-u_{k+1}}
+
\norm{X_-u_{k+1}}^2
\right)
\end{split}
\end{equation}
We sum equations~\ref{eq:proof-of-convergence6} and~\ref{eq:proof-of-convergence7} to get
\begin{equation}
\begin{split}
\norm{u}_{\HS{1}{0}{SM}}^2
&=
\norm{Xu}^2
+
\norm{\grad{h}u}^2
\\
&=
\sum_{k=0}^\infty
\left(1+\frac{k+n-2}{k}\right)
\norm{X_+u_{k-1}}^2
\\
&\quad\quad\quad+
\sum_{k=0}^\infty
\left(1+\frac{k}{k+n-2}\right)
\norm{X_-u_{k+1}}^2
+
\norm{W(u)}^2
\\
&\ge
\sum_{k=0}^\infty
\norm{X_+u_{k-1}}^2
+
\sum_{k=0}^\infty
\norm{X_-u_{k+1}}^2,
\end{split}
\end{equation}
which is estimate~\eqref{eq:xpm-convergence}.

\section{Proof of Liouville's theorem}
\label{sec:liouville}

This appendix is devoted to the proof of lemma~\ref{lma:liouville}. We let~$M$ be a compact smooth manifold with a smooth boundary. Suppose that we are given two~$C^{1,1}$ Riemannian metrics~$g$ and~$h$ on~$M$.
Let the corresponding unit sphere bundles be~$S_gM$ and~$S_hM$.
There is a natural radial~$C^{1,1}$-diffeomorphism~$(x,v) \mapsto (x,v\abs{v}^{-1}_h)$ from~$S_gM$ to~$S_hM$, the inverse map from~$S_hM$ to~$S_gM$ being~$(x,w) \mapsto (x,w\abs{w}^{-1}_g)$.

In the proof of lemma~\ref{lma:liouville} we use three types of Riemannian metrics on~$M$.
We will have a~$C^{1,1}$ Riemannian metric~$g$ and two types of smooth Riemannian metrics~$h$ and~$\alf{g}$.
We denote the corresponding radial diffeomorphisms by
\begin{equation}
\alf{s} \colon S_hM \to \alf{S}M,
\quad
s \colon S_hM \to S_gM,
\quad\text{and}\quad
\alf{r} \colon \alf{S}M \to S_gM.
\end{equation}
In the proof of lemma~\ref{lma:liouville} we will use the convention that the unit sphere bundle related~$\alf{g}$ is denoted~$\alf{S}M \coloneqq S_{\alf{g}}M$, the operators and differential forms related to~$\alf{g}$ are decorated with~$\alpha$ on top or as a subscript, the sphere bundle, operators and differential forms related to~$h$ are decorated with subscripts~$h$ and the bundles and the operators related to the metric~$g$ are written without decorations.

\begin{proof}[Proof of lemma~\ref{lma:liouville}]
The proof is based on smooth approximations of the Riemannian metric~$g$.
Let~$h$ be a smooth fixed reference Riemannian metric on~$M$.
Let~$\left(\alf{g}\right)$ be a sequence of smooth Riemannian metrics on~$M$ so that
\begin{equation}
\label{eq:g-limits}
\alf{g}_{jk}
\to
g_{jk}
\quad
\text{in }
W^{1,\infty}_h(M)
\quad
\text{and}
\quad
\alf{\Gamma}^{i}_{\ jk}
\to
\Gamma^i_{\ jk}
\quad
\text{in }
L^\infty_h(M).
\end{equation}
Existence of such sequence was proved in~\cite[Lemma 18]{IKPIXRTMLR}.
Let~$u \in \Lip(SM)$ and denote~$\alf{u} \coloneqq \alf{r}^\ast u$ and $\tilde u \coloneqq s^\ast u$.
We note that~$\tilde u = \alf{s}^\ast\alf{u}$.
We will prove that
\begin{equation}
\label{eq:limit}
\lim_{\alpha \to \infty}
\int_{\alf{S}M}
\alf{u}
L_{\alf{X}}(\d\alf{\Sigma})
=
\int_{SM}
u
L_{X}(\d\Sigma).
\end{equation}
Establishing equation~\eqref{eq:limit} proves the claim, since by Liouville's theorem~\cite[Lemma 3.6.4.]{PSUGIPETD} we have
\begin{equation}
L_{\alf{X}}(\d\alf{\Sigma}) = 0
\end{equation}
for all~$\alpha \in \N$ and thus the limit integral in equation~\eqref{eq:limit} is zero.

Recall that~$\tilde u = s^\ast u = \alf{s}^\ast\alf{u}$. Thus by basic properties of pullback it is enough prove that
\begin{equation}
\label{eq:limit-on-shm}
\lim_{\alpha \to \infty}
\int_{S_hM}
\tilde u
\alf{s}^{\ast}(L_{\alf{X}}\d\alf{\Sigma})
=
\int_{S_hM}
\tilde u
s^{\ast}(L_{X}\d\Sigma)
\end{equation}
The manifold~$M$ is the Euclidean unit ball in~$\R^n$ and we let~$(x^1,\dots,x^n)$ be usual Cartesian coordinates on~$M$.
We consider coordinates~$(x^1,\dots,x^n,w^1,\dots,w^n)$ on~$S_hM$ and corresponding coordinates
\[
(x^1,\dots,x^n,\alf{v}^1,\dots,\alf{v}^n)
\quad\text{on } \alf{S}M
\quad\text{and}\quad
(x^1,\dots,x^n,v^1,\dots,v^n)
\quad\text{on } SM
\]
so that~$\alf{s}(x,w) = (x,\alf{v})$ and~$s(x,w) = (x,v)$.
We associate to~$(x,w)$ the coordinate vector fields $\partial_{x^1},\dots,\partial_{x^n},\partial_{w^1},\dots,\partial_{w^n}$ and similarly $\partial_{x^1},\dots,\partial_{x^n},\partial_{\alf{v}^1},\dots,\partial_{\alf{v}^n}$ and $\partial_{x^1},\dots,\partial_{x^n},\partial_{v^1},\dots,\partial_{v^n}$ are associated to~$(x,\alf{v})$ and~$(x,v)$.
We let
\begin{equation}
\begin{split}
&\d x^1,\dots,\d x^n,\d w^1,\dots,\d w^n,
\\
&\d x^1,\dots,\d x^n,\d \alf{v}^1,\dots,\d \alf{v}^n,
\quad\text{and}
\\
&\d x^1,\dots,\d x^n,\d v^1,\dots,\d v^n
\end{split}
\end{equation}
be the dual basis one-forms characterized by
\begin{equation}
\begin{split}
&\d x^j(\partial_{x^k}) = \delta^{j}_k,
\quad
\d x^j(\partial_{w^k}) = 0,
\quad
\d w^j(\partial_{x^k}) = 0,
\quad
\d w^j(\partial_{w^k}) = \delta^{j}_k,
\\
&\d x^j(\partial_{x^k}) = \delta^{j}_k,
\quad
\d x^j(\partial_{\alf{v}^k}) = 0,
\quad
\d \alf{v}^j(\partial_{x^k}) = 0,
\quad
\d \alf{v}^j(\partial_{\alf{v}^k}) = \delta^{j}_k,
\\
&\d x^j(\partial_{x^k}) = \delta^{j}_k,
\quad
\d x^j(\partial_{v^k}) = 0,
\quad
\d v^j(\partial_{x^k}) = 0,
\quad
\d v^j(\partial_{v^k}) = \delta^{j}_k.
\end{split}
\end{equation}
Next, we will write the integrals in equation~\eqref{eq:limit-on-shm} in coordinates on~$S_hM$ and we will argue that equation~\eqref{eq:limit-on-shm} follows from~\eqref{eq:g-limits}.
We will derive a local coordinate formula for~$L_X(\d\Sigma)$.
A similar formula for~$L_{\alf{X}}(\d\alf{\Sigma})$ can be derived analogously.
Then we will compute how the coordinate presentations transform under the pullbacks~$s^\ast$ and~$\alf{s}^\ast$.

We denote by~$\abs{g}$ the determinant of~$g$. Since~$\d\Sigma$ is a volume form (differential form of the highest order), Cartan's formula implies that
\begin{equation}
L_X(\d\Sigma)
=
d(i_X\d\Sigma).
\end{equation}
Since
\begin{equation}
i_X\d x^i
=
\d x^i(X)
=
\d x^i(v^j\partial_{x^j} - \Gamma^l_{\ jk}v^jv^k\partial_{v^{l}})
=
v^i
\end{equation}
and
\begin{equation}
i_X\d v^i
=
\d v^i(X)
=
\d v^i(v^j\partial_{x^j} - \Gamma^l_{\ jk}v^jv^k\partial_{v^{l}})
=
-\Gamma^i_{\ jk}v^jv^k
\end{equation}
we see that
\begin{equation}
\label{eq:contraction}
\begin{split}
i_X\d\Sigma
&=
\sum_{i = 1}^nv^i\abs{g}
\,\d x^1
\wedge \dots \wedge
\widehat{\d x^i}
\wedge \dots \wedge
\d x^n
\wedge
\d v^1
\wedge \cdots \wedge
\d v^n
\\
&\quad
+
\sum_{i = 1}^n(-\Gamma^i_{\ jk}v^jv^k\abs{g})
\,\d x^1
\wedge \dots \wedge
\d x^n
\wedge
\d v^1
\wedge \cdots \wedge
\widehat{\d v^i}
\wedge \cdots \wedge
\d v^n,
\end{split}
\end{equation}
where~$\widehat{\d x^i}$ and~$\widehat{\d v^i}$ indicate that one-forms~$\d x^i$ and~$\d v^i$ are omitted from the wedge product.
From~\eqref{eq:contraction} it follows that
\begin{equation}
\label{eq:lxg}
\begin{split}
d(i_X\d\Sigma)
&=
\sum_{i = 1}^n
(-1)^{i-1}
\partial_{x^i}(v^i\abs{g})
\,\d x^1
\wedge \dots \wedge
\d x^n
\wedge
\d v^1
\wedge \cdots \wedge
\d v^n
\\
&\quad
+
\sum_{i = 1}^n
(-1)^{n+i-1}\partial_{v^i}(-\Gamma^i_{\ jk}v^jv^k\abs{g})
\,\d x^1
\wedge \dots \wedge
\d x^n
\wedge
\d v^1
\wedge \cdots \wedge
\d v^n
\\
&=
\sum_{i=1}^n
(-1)^{i-1}
(\partial_{x^i}(v^i\abs{g})
+
(-1)^{n+1}
\partial_{v^i}(\Gamma^i_{\ jk}v^jv^k\abs{g}))
\\
&\hspace{3em}\times
\,\d x^1
\wedge \dots \wedge
\d x^n
\wedge
\d v^1
\wedge \cdots \wedge
\d v^n.
\end{split}
\end{equation}
Similarly, we see that
\begin{equation}
\label{eq:lxalfg}
\begin{split}
L_{\alf{X}}(\d\alf{\Sigma})
&=
\sum_{i = 1}^n
(-1)^{i-1}
\partial_{x^i}(\alf{v}^i\abs{\alf{g}})
\,\d x^1
\wedge \dots \wedge
\d x^n
\wedge
\d \alf{v}^1
\wedge \cdots \wedge
\d \alf{v}^n
\\
&\quad
+
\sum_{i = 1}^n(-1)^{n+i-1}
\partial_{\alf{v}^i}(-\alf{\Gamma}^i_{\ jk}\alf{v}^j\alf{v}^k\abs{\alf{g}})
\,\d x^1
\wedge \dots \wedge
\d x^n
\wedge
\d \alf{v}^1
\wedge \cdots \wedge
\d \alf{v}^n
\\
&=
\sum_{i=1}^n
(-1)^{i-1}
(\partial_{x^i}(\alf{v}^i\abs{\alf{g}})
+
(-1)^{n+1}
\partial_{\alf{v}^i}(\alf{\Gamma}^i_{\ jk}\alf{v}^j\alf{v}^k\abs{\alf{g}}))
\\
&\hspace{3em}\times
\,\d x^1
\wedge \dots \wedge
\d x^n
\wedge
\d \alf{v}^1
\wedge \cdots \wedge
\d \alf{v}^n.
\end{split}
\end{equation}

Next, we pullback formulas~\eqref{eq:lxg} and~\eqref{eq:lxalfg} onto~$S_hM$.
We can compute
\begin{equation}
s^{\ast}\d v^j
=
\d(s^\ast v^j)
=
\d(w^j\abs{w}^{-1}_g)
=
\abs{w}^{-1}_g\d w^j
+
w^j\d(\abs{w}^{-1}_g).
\end{equation}
If we write
\begin{equation}
\d(\abs{w}^{-1}_g)
=
\mu_i\d x^i
+
\lambda_i\d w^i,
\end{equation}
then
\begin{equation}
\mu_k
=
\mu_i\d x^i(\partial_{x^k})
=
\d(\abs{w}^{-1}_g)(\partial_{x^k})
=
\partial_{x^k}\abs{w}^{-1}_g
\quad
\text{and}
\quad
\lambda_k = \partial_{w^k}\abs{w}^{-1}_g.
\end{equation}
Thus
\begin{equation}
s^\ast\d v^j
=
w^j(\partial_{x^k}\abs{w}^{-1}_g)\d x^k
+
(\abs{w}^{-1}_g\delta^j_k
+ 
w^j\partial_{w^k}\abs{w}^{-1}_g)\d w^k.
\end{equation}
Similarly we get
\begin{equation}
s^\ast\d \alf{v}^j
=
w^j(\partial_{x^k}\abs{w}^{-1}_\alpha)\d x^k
+
(\abs{w}^{-1}_\alpha\delta^j_k
+ 
w^j\partial_{w^k}\abs{w}^{-1}_\alpha)\d w^k.
\end{equation}
Since~$s$ and~$\alf{s}$ act identically on the base point~$x$, we have
\begin{equation}
s^\ast(\d x^1 \wedge \cdots \wedge \d x^n)
=
\d x^1 \wedge \cdots \wedge \d x^n
\quad\text{and}\quad
\alf{s}^\ast(\d x^1 \wedge \cdots \wedge\d x^n)
=
\d x^1 \wedge \cdots \wedge \d x^n.
\end{equation}
Using the fact that a wedge product vanishes whenever repetition appears we get
\begin{equation}
\begin{split}
&
s^\ast(
\d x^1
\wedge\dots\wedge
\d x^n
\wedge
\d v^1
\wedge\dots\wedge
\d v^n)
\\
&\quad=
\d x^1
\wedge\dots\wedge
\d x^n
\wedge
(\abs{w}^{-1}_g\delta^j_k + w^1(\partial_{w^k}\abs{w}^{-1}_g))\d w^k
%\\
%&\qquad\quad
\wedge
\cdots
\\&\qquad
\cdots
\wedge
(\abs{w}^{-1}_g\delta^n_k + w^n(\partial_{w^k}\abs{w}^{-1}_g))\d w^k
\\
&\quad=
\d x^1
\wedge\dots\wedge
\d x^n
\wedge
%\\
%&
%\qquad
\bigwedge_{j=1}^n
(\abs{w}^{-1}_g\delta^j_k + w^j(\partial_{w^k}\abs{w}^{-1}_g))\d w^k.
\end{split}
\end{equation}
By a similar computation
\begin{equation}
\begin{split}
&
s^\ast(
\d x^1
\wedge\dots\wedge
\d x^n
\wedge
\d v^1
\wedge\dots\wedge
\d v^n)
\\&\quad
=
\d x^1
\wedge\dots\wedge
\d x^n
\wedge
%\\
%&\qquad
\bigwedge_{j=1}^n
(\abs{w}^{-1}_g\delta^j_k + w^j(\partial_{w^k}\abs{w}^{-1}_\alpha))\d w^k.
\end{split}
\end{equation}
To complete formulas for the pullback of~\eqref{eq:lxg} and~\eqref{eq:lxalfg} we use the facts that~$s^\ast = s^{-1}_\ast$ and~$\alf{s}^\ast = \alf{s}^{-1}_\ast$ to compute
\begin{equation}
s^\ast \partial_{x^i}
=
\partial_{x^i}
+
(\partial_{x^i}w^j)\partial_{w^j}
\quad\text{and}\quad
\alf{s}^\ast \partial_{x^i}
=
\partial_{x^i}
+
(\partial_{x^i}w^j)\partial_{w^j}
\end{equation}
as well as
\begin{equation}
s^\ast\partial_{v^i}
=
(\partial_{v^j}w^j)\partial_{w^j}
\quad\text{and}\quad
\alf{s}^\ast\partial_{\alf{v}^i}
=
(\partial_{\alf{v}^j}w^j)\partial_{w^j}.
\end{equation}
Thus we get
\begin{align}
s^\ast(\partial_{x^i}v^i\abs{g})
&=
\partial_{x^i}(w^i\abs{w}^{-1}_g\abs{g})
+
(\partial_{x^i}w^j)(\partial_{w^j}(w^i\abs{w}^{-1}_g\abs{g})),
\\
\alf{s}^\ast(\partial_{x^i}v^i\abs{\alf{g}})
&=
\partial_{x^i}(w^i\abs{w}^{-1}_\alpha\abs{\alf{g}})
+
(\partial_{x^i}w^j)(\partial_{w^j}(w^i\abs{w}^{-1}_\alpha\abs{\alf{g}})),
\end{align}
and
\begin{align}
s^\ast\partial_{v^i}(\Gamma^i_{\ jk}v^jv^k\abs{g})
&=
\Gamma^i_{\ jk}\abs{g}
(\partial_{v^i}w^l)
\partial_{w^l}(w^i\abs{w}^{-1}_gw^k\abs{w}^{-1}_g),
\\
\alf{s}^\ast\partial_{\alf{v}^i}
(\alf{\Gamma}^i_{\ jk}\alf{v}^j\alf{v}^k\abs{\alf{g}})
&=
\alf{\Gamma}^i_{\ jk}\abs{\alf{g}}
(\partial_{\alf{v}^i}w^l)
\partial_{w^l}(w^i\abs{w}^{-1}_\alpha w^k\abs{w}^{-1}_\alpha).
\end{align}
The formulas we get for the pullbacks of~$L_X(\d\Sigma)$ along~$s$ and of~$L_{\alf{X}}(\d\alf{\Sigma})$ along~$\alf{s}$ are
\begin{equation}
\label{eq:pullback-g}
\begin{split}
s^\ast
L_X(\d\Sigma)
&=
\sum_{i=1}^n
(-1)^{i-1}
\bigg(
\partial_{x^i}(w^i\abs{w}^{-1}_g\abs{g})
+
(\partial_{x^k}w^j)
(\partial_{w^j}(w^k\abs{w}^{-1}_g\abs{g}))
\\
&\qquad\qquad+
(-1)^{n+1}
\Gamma^i_{\ jk}\abs{g}(\partial_{v^m}w^l)
\partial_{w^l}(w^m\abs{w}^{-1}_gw^k\abs{w}^{-1}_g)
\bigg)
\\
&\qquad\qquad
\d x^1 \wedge \cdots \wedge \d x^n \wedge
\bigwedge_{j=1}^n
(
\abs{w}^{-1}_g\delta^j_k
+
w^j(\partial_{w^k}\abs{w}^{-1}_g)
)
\d w^k
\end{split}
\end{equation}
and
\begin{equation}
\label{eq:pullback-alfg}
\begin{split}
\alf{s}^\ast
L_{\alf{X}}(\d\alf{\Sigma})
&=
\sum_{i=1}^n
(-1)^{i-1}
\bigg(
\partial_{x^i}(w^i\abs{w}^{-1}_\alpha\abs{\alf{g}})
+
(\partial_{x^k}w^j)
(\partial_{w^j}(w^k\abs{w}^{-1}_\alpha\abs{\alf{g}}))
\\
&\qquad\qquad+
(-1)^{n+1}
\alf{\Gamma}^i_{\ jk}\abs{\alf{g}}(\partial_{\alf{v}^m}w^l)
\partial_{w^l}(w^m\abs{w}^{-1}_\alpha w^k\abs{w}^{-1}_\alpha)
\bigg)
\\
&\qquad\qquad
\d x^1 \wedge \cdots \wedge \d x^n \wedge
\bigwedge_{j=1}^n
(
\abs{w}^{-1}_\alpha\delta^j_k
+
w^j(\partial_{w^k}\abs{w}^{-1}_\alpha)
)
\d w^k.
\end{split}
\end{equation}
From formulas~\eqref{eq:pullback-g} and~\eqref{eq:pullback-alfg} we see that can conclude the equation~\eqref{eq:limit-on-shm} if the following holds:
\begin{equation}
\label{eq:final-limit1}
\begin{split}
\partial_{x^i}(w^i\abs{w}^{-1}_\alpha\abs{\alf{g}})
&\prod_{j \in S}(\abs{w}^{-1}_\alpha \delta^j_k)
\prod_{j \in S'}(w^j(\partial_{w^k}\abs{w}^{-1}_\alpha))
\\
&\to
\partial_{x^i}(w^i\abs{w}^{-1}_\alpha\abs{g})
\prod_{j \in S}(\abs{w}^{-1}_g \delta^j_k)
\prod_{j \in S'}(w^j(\partial_{w^k}\abs{w}^{-1}_g)),
\end{split}
\end{equation}
\begin{equation}
\label{eq:final-limit2}
\begin{split}
(\partial_{x^i}w^j)(\partial_{w^j}(w^k&\abs{w}^{-1}_\alpha\abs{\alf{g}}))
\prod_{j \in S}(\abs{w}^{-1}_\alpha \delta^j_k)
\prod_{j \in S'}(w^j(\partial_{w^k}\abs{w}^{-1}_\alpha))
\\
&\to
(\partial_{x^i}w^j)(\partial_{w^j}(w^k\abs{w}^{-1}_g\abs{g}))
\prod_{j \in S}(\abs{w}^{-1}_g \delta^j_k)
\prod_{j \in S'}(w^j(\partial_{w^k}\abs{w}^{-1}_g)),
\end{split}
\end{equation}
\begin{equation}
\label{eq:final-limit3}
\begin{split}
&\alf{\Gamma}^i_{\ jk}\abs{\alf{g}}(\partial_{\alf{v}^m}w^l)
(\partial_{w^l}(w^m\abs{w}^{-1}_\alpha w^l\abs{w}^{-1}_\alpha))
\prod_{j \in S}(\abs{w}^{-1}_\alpha \delta^j_k)
\prod_{j \in S'}(w^j(\partial_{w^k}\abs{w}^{-1}_\alpha))
\\
&\to
\Gamma^i_{\ jk}\abs{g}(\partial_{v^m}w^l)
(\partial_{w^l}(w^m\abs{w}^{-1}_g w^l\abs{w}^{-1}_g))
\prod_{j \in S}(\abs{w}^{-1}_g \delta^j_k)
\prod_{j \in S'}(w^j(\partial_{w^k}\abs{w}^{-1}_g))
\end{split}
\end{equation}
in~$L^1(S_hM)$, where~$S$ and~$S'$ are any subsets of~$\{1,\dots,n\}$.
We chose the approximating sequence~$\left(\alf{g}\right)$ so that
\begin{equation}
\label{eq:limits-g-recall}
\alf{g}_{jk}
\to
g_{jk}
\quad
\text{in }
W^{1,\infty}_h(M)
\quad
\text{and}
\quad
\alf{\Gamma}^{i}_{\ jk}
\to
\Gamma^i_{\ jk}
\quad
\text{in }
L^\infty_h(M).
\end{equation}
From~\eqref{eq:limits-g-recall} we see that
\begin{equation}
\begin{split}
&\partial_{x^i}(w^i\abs{w}^{-1}_\alpha\abs{\alf{g}})
\to
\partial_{x^i}(w^i\abs{w}^{-1}_g\abs{g}),
\\
&\abs{w}^{-1}_\alpha\delta^{j}_k
\to
\abs{w}^{-1}_g\delta^{j}_k,
\\
&w^j(\partial_{w^k}\abs{w}^{-1}_\alpha))
\to
w^j(\partial_{w^k}\abs{w}^{-1}_g)),
\\
&\partial_{w^j}(w^k\abs{w}^{-1}_\alpha\abs{\alf{g}}
\to
\partial_{w^j}(w^k\abs{w}^{-1}_g\abs{g},
\\
&\alf{\Gamma}^i_{\ jk}\abs{\alf{g}}
\to
\Gamma^i_{\ jk}\abs{g},
\\
&\partial_{\alf{v}^m}w^l
\to
\partial_{v^m}w^l,
\\
&\partial_{w^l}(w^m\abs{w}_\alpha^{-1}w^l\abs{w}^{-1}_\alpha)
\to
\partial_{w^l}(w^m\abs{w}_g^{-1}w^l\abs{w}^{-1}_g)
\end{split}
\end{equation}
in~$L^\infty(S_hM)$.
Thus we can take products and we conclude that~\eqref{eq:final-limit1},~\eqref{eq:final-limit2} and~\eqref{eq:final-limit3} hold, which finishes the proof.
\end{proof}

\bibliographystyle{alpha}
\bibliography{references}

\begin{thebibliography}{GPSU16}

\bibitem[AD18]{ADXRTGFCFS}
Yernat~M. Assylbekov and Nurlan~S. Dairbekov.
\newblock The {X}-ray transform on a general family of curves on {F}insler
  surfaces.
\newblock {\em J. Geom. Anal.}, 28(2):1428--1455, 2018.

\bibitem[Ain13]{AinsworthAMRTS}
Gareth Ainsworth.
\newblock The attenuated magnetic ray transform on surfaces.
\newblock {\em Inverse Probl. Imaging}, 7(1):27--46, 2013.

\bibitem[AR97]{ARUDFFDIAG}
Yu.~E. Anikonov and V.~G. Romanov.
\newblock On uniqueness of determination of a form of first degree by its
  integrals along geodesics.
\newblock {\em J. Inverse Ill-Posed Probl.}, 5(6):487--490, 1997.

\bibitem[BI10]{BIBRFVMMCFO}
Dmitri Burago and Sergei Ivanov.
\newblock Boundary rigidity and filling volume minimality of metrics close to a
  flat one.
\newblock {\em Ann. of Math. (2)}, 171(2):1183--1211, 2010.

\bibitem[Cro90]{CRSNPC}
Christopher~B. Croke.
\newblock Rigidity for surfaces of nonpositive curvature.
\newblock {\em Comment. Math. Helv.}, 65(1):150--169, 1990.

\bibitem[Cro91]{CrokeRDBBP}
Christopher~B. Croke.
\newblock Rigidity and the distance between boundary points.
\newblock {\em J. Differential Geom.}, 33(2):445--464, 1991.

\bibitem[CS98]{CSSRCNCM}
Christopher~B. Croke and Vladimir~A. Sharafutdinov.
\newblock Spectral rigidity of a compact negatively curved manifold.
\newblock {\em Topology}, 37(6):1265--1273, 1998.

\bibitem[dHI17]{HIATLRAXTSSM}
Maarten~V. de~Hoop and Joonas Ilmavirta.
\newblock Abel transforms with low regularity with applications to x-ray
  tomography on spherically symmetric manifolds.
\newblock {\em Inverse Problems}, 33(12):124003, 36, 2017.

\bibitem[DPSU07]{DPSUBRPPMF}
Nurlan~S. Dairbekov, Gabriel~P. Paternain, Plamen Stefanov, and Gunther
  Uhlmann.
\newblock The boundary rigidity problem in the presence of a magnetic field.
\newblock {\em Adv. Math.}, 216(2):535--609, 2007.

\bibitem[DS03]{DSSPIGAM}
Nurlan~S. Dairbekov and Vladimir~A. Sharafutdinov.
\newblock Some problems of integral geometry on {A}nosov manifolds.
\newblock {\em Ergodic Theory Dynam. Systems}, 23(1):59--74, 2003.

\bibitem[DS10]{DSCKSTFRM}
N.~S. Dairbekov and V.~A. Sharafutdinov.
\newblock Conformal {K}illing symmetric tensor fields on {R}iemannian
  manifolds.
\newblock {\em Mat. Tr.}, 13(1):85--145, 2010.

\bibitem[FU01]{FUXRTNACTD}
David Finch and Gunther Uhlmann.
\newblock The x-ray transform for a non-abelian connection in two dimensions.
\newblock {\em Inverse Problems}, 17(4):695, aug 2001.

\bibitem[GMT21]{GMTBLRNCM}
Colin Guillarmou, Marco Mazzucchelli, and Leo Tzou.
\newblock Boundary and lens rigidity for non-convex manifolds.
\newblock {\em Amer. J. Math.}, 143(2):533--575, 2021.

\bibitem[GPSU16]{GPSXRTCNC}
Colin Guillarmou, Gabriel~P. Paternain, Mikko Salo, and Gunther Uhlmann.
\newblock The {X}-ray transform for connections in negative curvature.
\newblock {\em Comm. Math. Phys.}, 343(1):83--127, 2016.

\bibitem[Har50]{HartmanLUG}
Philip Hartman.
\newblock On the local uniqueness of geodesics.
\newblock {\em Amer. J. Math.}, 72:723--730, 1950.

\bibitem[IK23]{IKPIXRTMLR}
Joonas Ilmavirta and Antti Kykk\"{a}nen.
\newblock Pestov identities and {X}-ray tomography on manifolds of low
  regularity.
\newblock {\em Inverse Probl. Imaging}, 17(6):1301--1328, 2023.

\bibitem[Ilm18]{IlmavirtaXRTPR}
Joonas Ilmavirta.
\newblock X-ray transforms in pseudo-{R}iemannian geometry.
\newblock {\em J. Geom. Anal.}, 28(1):606--626, 2018.

\bibitem[IM19]{IMIGMBA}
Joonas Ilmavirta and Fran\c{c}ois Monard.
\newblock Integral geometry on manifolds with boundary and applications.
\newblock In {\em The {R}adon transform---the first 100 years and beyond},
  volume~22 of {\em Radon Ser. Comput. Appl. Math.}, pages 43--113. Walter de
  Gruyter, Berlin, [2019] \copyright 2019.

\bibitem[IM23]{IMGRTSSRFM}
Joonas Ilmavirta and Keijo M\"{o}nkk\"{o}nen.
\newblock The geodesic ray transform on spherically symmetric reversible
  {F}insler manifolds.
\newblock {\em J. Geom. Anal.}, 33(4):Paper No. 137, 27, 2023.

\bibitem[IP22]{IPBRTORO}
Joonas Ilmavirta and Gabriel~P. Paternain.
\newblock Broken ray tensor tomography with one reflecting obstacle.
\newblock {\em Comm. Anal. Geom.}, 30(6):1269--1300, 2022.

\bibitem[IS16]{ISBRTRSCO}
Joonas Ilmavirta and Mikko Salo.
\newblock Broken ray transform on a {R}iemann surface with a convex obstacle.
\newblock {\em Comm. Anal. Geom.}, 24(2):379--408, 2016.

\bibitem[Jol07a]{Jol20072}
A.~Jollivet.
\newblock On inverse scattering in electromagnetic field in classical
  relativistic mechanics at high energies.
\newblock {\em Asymptot. Anal.}, 55(1-2):103--123, 2007.

\bibitem[Jol07b]{Jol20071}
Alexandre Jollivet.
\newblock On inverse problems in electromagnetic field in classical mechanics
  at fixed energy.
\newblock {\em J. Geom. Anal.}, 17(2):275--319, 2007.

\bibitem[Lee13]{Lee2013}
John~M. Lee.
\newblock {\em Introduction to smooth manifolds}, volume 218 of {\em Graduate
  Texts in Mathematics}.
\newblock Springer, New York, second edition, 2013.

\bibitem[Leh16]{LehtonenGRTTDCHM}
Jere Lehtonen.
\newblock {T}he geodesic ray transform on two-dimensional {C}artan-{H}adamard
  manifolds, 2016.
\newblock arXiv:1612.04800 [math.DG].

\bibitem[LRS18]{LRSTTCHM}
Jere Lehtonen, Jesse Railo, and Mikko Salo.
\newblock Tensor tomography on {C}artan-{H}adamard manifolds.
\newblock {\em Inverse Problems}, 34(4):044004, 27, 2018.

\bibitem[LSU03]{LSUSBRRM}
Matti Lassas, Vladimir Sharafutdinov, and Gunther Uhlmann.
\newblock Semiglobal boundary rigidity for {R}iemannian metrics.
\newblock {\em Math. Ann.}, 325(4):767--793, 2003.

\bibitem[MNP21]{FNPCINNXT}
Fran\c{c}ois Monard, Richard Nickl, and Gabriel~P. Paternain.
\newblock Consistent inversion of noisy non-{A}belian {X}-ray transforms.
\newblock {\em Comm. Pure Appl. Math.}, 74(5):1045--1099, 2021.

\bibitem[MP11]{MPIPGD}
Will Merry and Gabriel Paternain.
\newblock {L}ecture notes: {I}nverse problems in geometry and dynamics.
\newblock \url{https://www.dpmms.cam.ac.uk/~gpp24/ipgd(3).pdf}, 2011.

\bibitem[MR78]{MRPFIRMDS}
R.~G. Muhometov and V.~G. Romanov.
\newblock On the problem of finding an isotropic {R}iemannian metric in an
  {$n$}-dimensional space.
\newblock {\em Dokl. Akad. Nauk SSSR}, 243(1):41--44, 1978.

\bibitem[Muh77]{MukhometovRPTDRM}
R.~G. Muhometov.
\newblock The reconstruction problem of a two-dimensional {R}iemannian metric,
  and integral geometry.
\newblock {\em Dokl. Akad. Nauk SSSR}, 232(1):32--35, 1977.

\bibitem[Muh81]{MukhometovOPRRM}
R.~G. Muhometov.
\newblock On a problem of reconstructing {R}iemannian metrics.
\newblock {\em Sibirsk. Mat. Zh.}, 22(3):119--135, 237, 1981.

\bibitem[Muk75]{MukhometovIKPSP}
R.G. Mukhometov.
\newblock {I}nverse kinematic problem of seismic on the plane.
\newblock {\em Mathematical Problems of Geophysics, Akad. Nauk. SSSR, Sibirsk.
  Otdel., Vychisl. Tsentr, Novosibirsk}, 6:243--252, 1975.

\bibitem[Muk77]{MukhometovRPTDRMIG}
R.G. Mukhometov.
\newblock {T}he reconstruction problem of a two-dimensional {R}iemannian
  metric, and integral geometry ({R}ussian).
\newblock {\em Dokl. Akad. Nauk SSSR}, 232(1):32–35, 1977.

\bibitem[Nov19]{NovikovNARTIT}
R.~G. Novikov.
\newblock Non-{A}belian {R}adon transform and its applications.
\newblock In {\em The {R}adon transform---the first 100 years and beyond},
  volume~22 of {\em Radon Ser. Comput. Appl. Math.}, pages 115--127. Walter de
  Gruyter, Berlin, [2019] \copyright 2019.

\bibitem[Pat99]{PaternainGF}
Gabriel~P. Paternain.
\newblock {\em Geodesic flows}, volume 180 of {\em Progress in Mathematics}.
\newblock Birkh\"{a}user Boston, Inc., Boston, MA, 1999.

\bibitem[PS88]{PSIGOTFMNC}
L.~N. Pestov and V.~A. Sharafutdinov.
\newblock Integral geometry of tensor fields on a manifold of negative
  curvature.
\newblock {\em Sibirsk. Mat. Zh.}, 29(3):114--130, 221, 1988.

\bibitem[PS22]{PSNAXRTS}
Gabriel~P. Paternain and Mikko Salo.
\newblock The non-{A}belian {X}-ray transform on surfaces, 2022.
\newblock to appear in J. Differ. Geom.

\bibitem[PSU12]{PSUARTCHF}
Gabriel~P. Paternain, Mikko Salo, and Gunther Uhlmann.
\newblock The attenuated ray transform for connections and {H}iggs fields.
\newblock {\em Geom. Funct. Anal.}, 22(5):1460--1489, 2012.

\bibitem[PSU13]{PSUTTS}
Gabriel~P. Paternain, Mikko Salo, and Gunther Uhlmann.
\newblock Tensor tomography on surfaces.
\newblock {\em Invent. Math.}, 193(1):229--247, 2013.

\bibitem[PSU14a]{PSUSRIDAS}
Gabriel~P. Paternain, Mikko Salo, and Gunther Uhlmann.
\newblock Spectral rigidity and invariant distributions on {A}nosov surfaces.
\newblock {\em J. Differential Geom.}, 98(1):147--181, 2014.

\bibitem[PSU14b]{PSUTTPC}
Gabriel~P. Paternain, Mikko Salo, and Gunther Uhlmann.
\newblock Tensor tomography: progress and challenges.
\newblock {\em Chinese Ann. Math. Ser. B}, 35(3):399--428, 2014.

\bibitem[PSU15]{PSUIDBTTT}
Gabriel~P. Paternain, Mikko Salo, and Gunther Uhlmann.
\newblock Invariant distributions, {B}eurling transforms and tensor tomography
  in higher dimensions.
\newblock {\em Math. Ann.}, 363(1-2):305--362, 2015.

\bibitem[PSU23]{PSUGIPETD}
Gabriel~P. Paternain, Mikko Salo, and Gunther Uhlmann.
\newblock {\em Geometric inverse problems---with emphasis on two dimensions},
  volume 204 of {\em Cambridge Studies in Advanced Mathematics}.
\newblock Cambridge University Press, Cambridge, 2023.
\newblock With a foreword by Andr\'{a}s Vasy.

\bibitem[PU05]{PUTDCSRMBR}
Leonid Pestov and Gunther Uhlmann.
\newblock Two dimensional compact simple {R}iemannian manifolds are boundary
  distance rigid.
\newblock {\em Ann. of Math. (2)}, 161(2):1093--1110, 2005.

\bibitem[SS18]{SSGLR}
Clemens S\"{a}mann and Roland Steinbauer.
\newblock On geodesics in low regularity.
\newblock {\em J. Phys. Conf. Ser.}, 968:012010, 14, 2018.

\bibitem[SU98]{SURMWSLG}
Plamen Stefanov and Gunther Uhlmann.
\newblock Rigidity for metrics with the same lengths of geodesics.
\newblock {\em Math. Res. Lett.}, 5(1-2):83--96, 1998.

\bibitem[SU00]{USDBRSRRSNFP}
Vladimir Sharafutdinov and Gunther Uhlmann.
\newblock On deformation boundary rigidity and spectral rigidity of
  {R}iemannian surfaces with no focal points.
\newblock {\em J. Differential Geom.}, 56(1):93--110, 2000.

\bibitem[SU05]{SUBRASFGSM}
Plamen Stefanov and Gunther Uhlmann.
\newblock Boundary rigidity and stability for generic simple metrics.
\newblock {\em J. Amer. Math. Soc.}, 18(4):975--1003, 2005.

\bibitem[SU11]{SUARTSS}
Mikko Salo and Gunther Uhlmann.
\newblock The attenuated ray transform on simple surfaces.
\newblock {\em J. Differential Geom.}, 88(1):161--187, 2011.

\bibitem[SUVZ19]{SUVZTTT}
Plamen Stefanov, Gunther Uhlmann, Andras Vasy, and Hanming Zhou.
\newblock Travel time tomography.
\newblock {\em Acta Math. Sin. (Engl. Ser.)}, 35(6):1085--1114, 2019.

\bibitem[Uhl14]{UhlmannIPSU}
Gunther Uhlmann.
\newblock Inverse problems: seeing the unseen.
\newblock {\em Bull. Math. Sci.}, 4(2):209--279, 2014.

\end{thebibliography}

\end{document}